\documentclass[a4paper,10pt]{amsart}
\usepackage[latin1]{inputenc}
\usepackage[T1]{fontenc}
\usepackage{amsmath}
\usepackage[initials,backrefs]{amsrefs}
\usepackage{amsthm}
\usepackage{latexsym}
\usepackage{amssymb}
\usepackage{mathtools}

\mathtoolsset{showonlyrefs}
\usepackage[all]{xy}
\usepackage{calc}
\usepackage{multicol}
\usepackage{color}
\usepackage[left=2.61cm,right=2.61cm,top=2.72cm,bottom=2.72cm]{geometry}
\usepackage{ulem}
\usepackage{hyperref}
\newtheorem{thm}{Theorem}[section]
\newtheorem{prop}[thm]{Proposition}

\newtheorem{lem}[thm]{Lemma}
\newtheorem{cor}[thm]{Corollary}

\newtheorem{rmq}[thm]{Remark}

\newcommand{\R}{\mathbb{R}}
\numberwithin{equation}{section}
\newcommand{\N}{\mathbb{N}}

\newcommand{\Bcal}{\mathcal{B}}

\newcounter{exercice}

\begin{document}
\title[Singular radial solutions for KS]{Singular radial solutions for the Keller-Segel equation in high dimension }

\thanks{D. Bonheure \& J.B. Casteras are supported by MIS F.4508.14 (FNRS), PDR T.1110.14F (FNRS); J.B. Casteras is supported by the
Belgian Fonds de la Recherche Scientifique -- FNRS;
D. Bonheure is partially supported by the project ERC Advanced Grant  2013 n. 339958: ``Complex Patterns for Strongly Interacting Dynamical Systems - COMPAT'' and by ARC AUWB-2012-12/17-ULB1- IAPAS. J. F\" oldes is partly supported by the National Sicence Foundation under the grant NSF-DMS-1816408. }

\author[Bonheure, Casteras, and F\" oldes]{Denis Bonheure \and Jean-Baptiste Casteras \and Juraj F\" oldes}

\address{Denis Bonheure, Jean-Baptiste Casteras
\newline \indent D\'epartement de Math\'ematiques, Universit\'e Libre de Bruxelles,
\newline \indent CP 214, Boulevard du triomphe, B-1050 Bruxelles, Belgium.}
\email{Denis.Bonheure@ulb.ac.be}
\email{jeanbaptiste.casteras@gmail.com}

\address{Juraj F\" oldes
\newline \indent Dept. of Mathematics, University of Virginia \indent 
\newline \indent P.O. Box 400137, 322 Kerchof Hall, Charlottesville, VA 22904-4137,\indent }
\email{foldes@virginia.edu}

\begin{abstract}
We study singular radially symmetric solution of the stationary Keller-Segel equation, that is, an elliptic equation with
exponential nonlinearity, which is super-critical in dimension $N \geq 3$. The solutions 
are unbounded at the origin and we show that they describe the asymptotics of bifurcation 
branches of regular solutions. It is shown that for any ball and any $k \geq 0$, there is a singular solution 
that satisfies Neumann boundary condition and oscillates at least $k$ times around the constant 
equilibrium. Moreover,  we prove that in dimension $3 \leq N \leq 9$ there are regular solutions 
satisfying Neumann boundary conditions that are close to singular ones. Hence, it follows that 
there exist regular solutions on any ball with arbitrarily fast oscillations. 
For generic radii, we show that the bifurcation branches of regular solutions oscillate in the bifurcation plane when $3\leq N\leq 9$
and approach to a singular solution. In dimension $N > 10$, we show that 
the Morse index of the singular solution is finite, and therefore the existence of regular solutions with fast oscillations 
is not expected. 
\end{abstract}

\maketitle

\section{Introduction}
The goal of the present paper is to investigate singular, radial solutions of the so-called Keller-Segel equation
\begin{equation}\label{eq:intro_v_lambda}
\left\{
\begin{aligned}
-\Delta v+v &= \lambda e^v& \qquad &\text{in } B_R \setminus\{0\} \,, \\
\quad v &> 0 & &\text{in }B_R \setminus\{0\} \,, \\
\partial_\nu v &= 0 &\quad &\text{on } \partial B_R\,,
\end{aligned}
\right. 
\end{equation}
where $B_R \subset \R^N$, $N \geq 3$  is a ball of radius $R > 0$ centered at the origin. The solutions are assumed to
blow-up at the origin with a specific rate (see \eqref{KSwhole} below) which is in some sense minimal so that they 
are limits of  sequences of 
regular solutions with value at the origin approaching infinity.  
Then,  qualitative properties of singular solutions 
such as Morse index, yield information about oscillations of the bifurcation branches. We give more details below. 

The problem \eqref{eq:intro_v_lambda} is motivated by models of chemotaxis,  an omnipresent mechanism in biology that describes the motion of species towards higher (lower) concentration of a chemical substance, for example nutrients or poisons. Sometimes the substance is 
also secreted by the species themselves, which induces a complicated large scale behavior such as aggregation, scattering, or pattern formation. 
Mathematically, this phenomenon can be described by a strongly coupled evolution system introduced by Keller and Segel \cite{Keller}
\begin{equation}\label{kssys}
\left\{
\begin{aligned}
\dfrac{\partial u}{\partial t} &= \Delta u - D_1\nabla \cdot (u\nabla \phi(v)) \quad &\text{in } \Omega \times (0, T),\\[3mm]  
\dfrac{\partial v}{\partial t} &= D_2 \Delta v-D_3 v+D_4 u \quad &\text{in } \Omega \times (0, T),\\[3mm] 
\end{aligned}
\right.
\end{equation}
where $T > 0$, $\Omega \subset \R^N$ is a smooth bounded domain, $D_i$, $i=1,\cdots ,4$ are positive constants, and $\phi$ is a smooth strictly increasing function, which depends on a particular model. Since function $v$ represents the concentration of a chemical substance and $u$ 
stands for the concentration of the considered organisms, it is natural to suppose 
\begin{equation}\label{posi}
u,v \geq 0 \qquad  \text{in } \Omega \times (0, T) 
\end{equation}
and to complement the model with no-flux boundary conditions
\begin{equation}\label{bc}
\partial_\nu u=\partial_\nu v=0 \quad \text{on } \partial\Omega \times (0, T), \\[3mm]
\end{equation}
and some non-negative initial conditions. 
 The system \eqref{kssys} has attracted a lot of attention these past decades and we refer to surveys \cite{MR2013508, MR2073515}, and to references therein for more details on the existence, blow-up, and asymptotic behavior  of solutions.

The analysis of  global dynamics  of \eqref{kssys} crucially depends on the understanding of equilibria, that is, solutions of 
$$
\nabla\cdot (u \nabla (\log u - D_1 \phi (v))=0, \qquad D_2 \Delta v - D_3v + D_4 u = 0 \,,
$$
with boundary conditions \eqref{bc}. 
By a standard reasoning one has 
$u=C e^{D_1 \phi (v)}$ for some positive constant $C$. The canonical choices for $\phi$ are $\phi (v)=v$, which leads to the Keller-Segel equation  \eqref{eq:intro_v_lambda} on a domain $\Omega$ and $\phi (v)=\ln v$, which after appropriate rescaling,  yields Lin-Ni-Takagi equation
\begin{equation}\label{lnt}
\left\{
\begin{aligned}
- \Delta v+v &=v^p& \qquad & \text{in }\Omega,\\ 
\quad v &>0 && \text{in }\Omega,\\
\partial_\nu v &=0 &&\text{on } \partial \Omega .
\end{aligned}
\right.
\end{equation}
 The constants $\lambda$ and $p$ in \eqref{eq:intro_v_lambda}  and \eqref{lnt} respectively depend on the parameters $D_i$ of the system. 
 A large amount of literature has been devoted to the Lin-Ni-Takagi equation in the subcritical and critical case, that is, 
 when $N\geq 3$ and $1<p \leq p_S := \dfrac{N+2}{N-2}$ (see \cite{MR1696122, del2015interior, MR1721719, MR929196} and references therein). 
Much less is known in the super-critical case, $p > p_S$ for \eqref{lnt} or  $N \geq 3$  for  \eqref{eq:intro_v_lambda}, see  
\cite{benguria, MR1769174, WangWei2002}.
 \smallskip

Clearly if $p$ increases, then the problem \eqref{lnt} becomes `more super-critical', however the role of $\lambda$ in \eqref{eq:intro_v_lambda} is less 
obvious, since the character of the nonlinearity remains unchanged as $\lambda$ varies. To obtain a better insight, notice that 
\eqref{lnt} has two constant equilibria $v \equiv 0$ and $v \equiv 1$ which are in particular independent of $p$. 
On the other hand if $\lambda < 1/e$, then \eqref{eq:intro_v_lambda} has two constant solutions $\underline{u}_\lambda < 1 < \bar{u}_\lambda$ satisfying
\begin{equation}\label{ceqs}
\lambda e^\mu=\mu 
\end{equation}
and if $\lambda > 1/e$ there is no constant solution. Furthermore, $\bar{u}_\lambda \to \infty$ and $\underline{u}_\lambda \to 0$ as $\lambda \to 0^+$. 
To reveal the analogy between \eqref{lnt} and \eqref{eq:intro_v_lambda} we denote $\mu = \bar{u}_\lambda$ and 
\begin{equation}
u:=\frac{v_\lambda}{\overline{u}_\lambda}=\frac{v_\lambda}{\mu} \,.
\end{equation}
Then, $u$ satisfies
\begin{equation}\label{eq:intro_u_mu}
\left\{
\begin{aligned}
-\Delta u+u &= e^{\mu(u-1)}&\qquad &\text{in } B_R \\
u &> 0 && \text{in }B_R,\\
\partial_\nu u &= 0   &\quad &\text{on } \partial B_R \,,
\end{aligned}
\right.
\end{equation}
with $u \equiv 1$ and $u \equiv \underline{u}_\mu$, where 
\begin{equation}\label{eq:underline_u_mu_def}
\underline{u}_\mu=e^{\mu(\underline{u}_\mu-1)}, \quad \underline{u}_\mu<1
\end{equation}
are constant solutions of \eqref{eq:intro_u_mu}. In this form it is more visible that the nonlinearity becomes `more critical'
if $\mu$ is large, which is equivalent to $\lambda$ being small. 

The following bifurcation result for \eqref{eq:intro_u_mu} with parameter $\mu$ was obtained in \cite{bocano}, 
see \cite{BonheureGrumiauTroestler2015} for an analogous  result for \eqref{lnt}. 
Note that for fixed parameters the radial solutions of the second order equations are uniquely determined by the value of the function at $0$ (since $u'(0) = 0$), 
therefore it suffices to investigate bifurcation diagrams in $\R^2$ with coordinates corresponding to $\mu$ and $u(0)$.
Specifically, by $(\mu_0, A)$ we denote a pair $(\mu_0, u)$, where $u$ is the solution of \eqref{eq:intro_u_mu} with $\mu = \mu_0$
and $A = u(0)$.
Here and below 
$\lambda_i^{rad}$ denotes the $i$-th eigenvalue of the operator $-\Delta+Id$ in
the ball $B_R := \{x \in \mathbb{R}^N : |x| < R\}$ with Neumann boundary conditions, restricted to the space of radial functions.

\begin{thm}\label{thm:bifurcation}
For every $i\geq 2$, the trivial branch $(\mu, 1)$ of problem \eqref{eq:intro_u_mu} 
has a  bifurcation point at $(\lambda_i^{rad},1)$. 
Let $\Bcal_i \subset \R^2$ be the continuum that branches out of $(\lambda_i^{rad},1)$. The following holds
\begin{itemize}
\item[(i)] the branches $\Bcal_i$ are unbounded and do not intersect, and furthermore close to $(\lambda_i^{rad},1)$, $\Bcal_i$ is a $C^1$-curve;
\item[(ii)] if $(\mu, A) \in \Bcal_i$, then the corresponding solution $u_\mu$ satisfies $u_\mu > 0$ in $B_R$;
\item[(iii)] each branch consists of two connected components $\Bcal_i^- := \Bcal_i \cap \{(\mu, A) : A < 1\}$ and 
 $\Bcal_i^+ := \Bcal_i \cap \{(\mu, A) : A > 1\}$;
\item[(iv)] if $(\mu, A) \in \Bcal_i$ then the corresponding $u_\mu-1$ has exactly $i-1$ zeros, $u_\mu'$ has exactly $i-2$ zeros;
\item[(v)] the functions satisfying $u_\mu(0) < 1$ are uniformly bounded in the $C^1$-norm.
\end{itemize}
\end{thm}

The above theorem guarantees that $\Bcal_i^-$ is a subset $\R \times (0, 1)$ and it is unbounded. Since there are no non-trivial solutions for $\mu \leq 0$, 
we obtain that for each $i \geq 2$ the curve $\Bcal_i^-$ is unbounded from above in the $\mu$ coordinate.  
We refer an interested reader to \cite{bocano2,bocano,del2014large,PistoiaVaira2015} for the construction of solutions that we expect to be on the lower branches $\mathcal{B}_i^-$ (the solutions lie in the half plane $\{u(0) < 1\}$, but it is not known whether they are connected with the trivial solution). 
Note that all the references above except \cite{del2014large} deal with radial solutions and  analogous results 
to Theorem \ref{thm:bifurcation} for the Lin-Ni-Takagi equation has been proved in \cite{BonheureGrossiNorisTerracini2015}. 
We also refer to \cite{BoCoNo,CoNo} for related problems involving the $p$-laplace operator. 

Properties of the  upper branches $\mathcal{B}_i^+$ are more delicate, since the corresponding solutions are not 
a priori uniformly bounded. 
Although our interest is in dimension $N \geq 3$, we first recall known results in two dimension.

If $N = 2$, then we call the problem 'critical'
since the exponential nonlinearity is critical. 
 It is proved in \cite{bocano} that the branches $\mathcal{B}_i^+$ are unbounded and they exist for all values of $\mu\geq \lambda_i^{rad}$. 
Since $\lambda \to 0$ as $\mu\to \infty$, this means that in $(\lambda, u(0))$ plane, $\mathcal{B}_i^+$ approaches arbitrary close to the 
line $\lambda = 0$.  
 Moreover, for $N = 2$ del Pino and Wei \cite{MR2209293} constructed a class of radial solutions $(u_\lambda)_{\lambda \ll 1}$ of \eqref{eq:intro_u_mu} such that 
\begin{equation}
u_\lambda (x) \rightarrow 8\pi \mathcal{G}(x,0) \qquad \textrm{as } \lambda \to 0^+
\end{equation} 
uniformly on a compact subsets of $B_R \setminus \{0\}$, where $\mathcal{G}$ is the Green's function, that is, for any $y\in \overline{B}_R$, 
$x \mapsto \mathcal{G}(x,y)$ solves
$$
-\Delta_x \mathcal{G}+\mathcal{G}=\delta_y \text{ in } B_R,\quad 
\dfrac{\partial \mathcal{G}}{\partial \nu_x}=0 \text{ on }\partial B_R 
$$
 and $\delta_y$ is the Dirac measure supported at $y$.  We remark that in \cite{MR2209293}, a result for non-radial solutions on general domains is also proved. 
 Since one can check that $w_\lambda (0)=u_\lambda (0)/ \overline{u}_\lambda >1$, the functions $(w_\lambda)_{\lambda > 0}$ belong to solutions in the upper half plane, 
 and their oscillation properties indicate that $\lambda \mapsto w_\lambda$ corresponds to the asymptotic part of the  first upper branch $\mathcal{B}_1^+$.
 The results of \cite{MR2209293} were extended, by the first two authors in collaboration with Rom\' an in  \cite{bocaro}, to 
  solutions concentrating on the boundary and/or on an interior sphere and blowing-up at the origin.  Even more generally, 
under  suitable non-degeneracy assumptions, it is possible to 
  show the existence of solutions $(v_\lambda)_{\lambda > 0}$ such that $v_\lambda(0) \to \infty$ as $\lambda \to 0^+$ and for every $M \geq 0$ there is 
 $(r_j)_{j = 1}^M \subset (0, R)$ such that $v_\lambda (r_j) \to \infty$ as $\lambda \to 0^+$. 
 These non-degeneracy conditions are conjectured to hold, and it is believed that the solutions that concentrate on $i$ spheres form the asymptote of the upper branch $\mathcal{B}_i^+$. We remark that in the `asymptotically critical' case, $p \approx p_S$ for Lin-Ni-Takagi equation with $N = 3$,  Rey and Wei \cite{reywei} constructed a class of solutions that are  believed to form the asymptote of $\mathcal{B}_1^+$.

\medbreak 

Our main aim is to describe the purely supercritical upper branches of \eqref{eq:intro_v_lambda}, a problem that
recently attracted a lot 
of attention especially with Dirichlet boundary conditions
\begin{equation}
\label{eqjl}
\begin{cases}U_{rr}+\dfrac{N-1}{r}U_r + \lambda g(U)=0,\ 0<r<1,\\
U>0,\ 0<r<1,\\
U(1)=0,\ \end{cases}
\end{equation}
see \cite{buddnorbury,dolbeaultflores,guowei,josephlundgren,merlepeletier,miya2}. In \cite{josephlundgren}, see also 
\cite[Chapter 2]{MR2779463} for a recent survey, Joseph and Lundgren 
considered $g(U)=e^U$ and proved that the set of positive solutions to \eqref{eqjl} forms a curve $\gamma$ emanating from 
the trivial solution $U \equiv 0$, $\lambda = 0$.
 When $3\leq N\leq 9$, $\gamma$ has infinitely many turning points around $\lambda^\ast =2(N-2)$ and blows up at $\lambda^\ast$. The case $N=3$ was treated earlier by Gel'fand \cite{MR0153960}. 
 When $N\geq 10$, the branch consists of minimal solutions for $0<\lambda <\lambda^\ast$ 
 with an asymptote at $\lambda = \lambda^\ast$. 
 If $g(U)=(1+U)^p$, then in \cite{josephlundgren} a special exponent $p_{JL}$ was found, namely
\begin{equation}
\label{defpjl}
p_{JL}=\begin{cases} 1+\dfrac{4}{N-4-2\sqrt{N-1}},& \text{when}\ N\geq 11, \\ \infty ,& \text{when} \ 2\leq N\leq 10, \end{cases}
\end{equation}
and it was proved that when $p_S < p < p_{JL}$, the branch emanating from $(U,\lambda)=(0,0)$ has infinitely many turning points around $\lambda^\ast =\theta (N-2-\theta)$, $\theta = \dfrac{2}{p-1}$ and blows up at $\lambda^\ast$ (the singular solution is given by $U^\ast = r^{-\theta } -1$), whereas if $p\geq p_{JL}$, the branch exists for all $0<\lambda<\lambda^\ast$, does not oscillate 
and blows up at $\lambda^\ast$. These results were extended to more general nonlinearities, see 
for instance \cite{miya2}, where the author considered nonlinearlity of the form 
\begin{gather}
g(u)=e^u +h(u), 
\end{gather}
with $h$ being a smooth lower order term. 

For analogous Neumann problem we are only aware of \cite{miya2}, 
where the author studied the structure of positive radial solutions $u_\lambda$ of 
\begin{equation}\label{ncm}
\left\{
\begin{aligned}
- \Delta u_\lambda  &= \lambda(u_\lambda^p - u_\lambda) ,& \qquad &\text{in} \ B_1,\\
\partial_\nu u_\lambda &= 0, &\ &\text{on} \ \partial B_1,
\end{aligned}
\right.
\end{equation}
that bifurcate from the trivial solution $1$. The exponent $p > \frac{N+2}{N-2}$ is fixed here. Problem \eqref{ncm} as well as \eqref{eqjl}  
possesses a crucial scaling, that allows 
for exchange of the parameter $\lambda$ for the size of the domain. 
More precisely, if $u_\lambda(\cdot)$ solves an appropriate problem on $B_R$ with parameter $\lambda$, 
then $u_{\lambda}(\alpha \cdot)$ solves the same problem on $B_{R/\alpha}$ with parameter $\alpha^2 \lambda$. This 
property allows for a construction of explicit singular solutions as well as proofs of various important non-degeneracy properties. 

Neumann problems even with scale invariance are more complicated than Dirichlet ones since 
there might be several bifurcation branches that contain positive solutions. In fact we show below that
there are infinitely many such branches. Also, radial eigenfunctions of Laplacian with Neumann boundary conditions
correspond to large eigenvalues.

\medbreak

In our problem \eqref{eq:intro_v_lambda} due to the presence of the zero order term, we cannot rely on any scaling or 
transformation that removes the parameter $\lambda$ from the equation. 
Moreover, the constant equilibria depend on $\lambda$ and after appropriate normalization (cf. \eqref{eq:intro_u_mu}) the 
parameter appears in the exponent of the nonlinearity, which introduces a novel parameter dependent problem.

To study the behavior of radial solutions for fixed parameter $\lambda > 0$, 
we first show that as the value of a solution at the origin increases, 
it converges to a solution $U^*_\lambda$ satisfying the same problem with an explicit singularity at the origin. 
The existence and uniqueness of $U^*_\lambda$ is shown on $(0, \infty)$, and in order to prove the existence of 
singular solution on a finite interval with appropriate boundary conditions we first show that $U^*_\lambda$
has infinitely many critical points. In other words, we show that for fixed $\lambda$, the restriction of 
$U^*_\lambda$ satisfies Neumann problem on infinitely many balls. More precisely, we prove that $U^*_\lambda$
oscillates around $\bar{u}_\lambda$. 

Before we formulate our first result, let us recall that $\bar{u}_\lambda$ is the largest solution of $u = \lambda e^u$. 

\begin{thm}
\label{KSsingintro}
For any $N\geq 3$ and $\lambda > 0$,  
there exists $U^\ast_\lambda = U^\ast > 0$ satisfying, for each $\delta \in (0, 1)$,
\begin{equation}
\label{KSwhole}
\left\{
\begin{aligned}
&- u^{\prime \prime}- \dfrac{N-1}{r}u^\prime +u = \lambda e^u &\qquad &\text{on} \ \R^+\\ 
&u(r) = -2\ln r +\ln \dfrac{2(N-2)}{\lambda}+ O(r^{2\delta}) & \qquad \ &\text{when} \ r\rightarrow 0. 
\end{aligned}
\right.
\end{equation}
Moreover, a solution satisfying the equation in \eqref{KSwhole} with boundary conditions 
\begin{equation}\label{gasy}
u(r) = -2\ln r +\ln \dfrac{2(N-2)}{\lambda}+ O(1)
\end{equation} 
is unique. 
In addition, if
\begin{equation}
\lambda < \lambda_N^* :=
\begin{cases}
0.16 & N = 3 \,,\\
0.35 & N = 4 \,, \\
0.36 & N = 5 \,, \\
\frac{1}{e} & N > 5 \,,
\end{cases}
\end{equation}
then $U^\ast$ attains infinitely many times the value $\overline{u}_\lambda$.
Furthermore, if there are sequences $(\gamma_n)_{n = 1}^\infty$ and $(\lambda_n)_{n = 1}^\infty$ with
$\gamma_n \rightarrow \infty$ and $\lambda_n \rightarrow \lambda_\infty \in (0, \infty)$, 
then
$u_n \rightarrow U^\ast$ in $C^0_{loc}((0,\infty))$, where $u_n$
is the solution to
\begin{equation}
\label{KSeqintfin}
\left\{
\begin{gathered}
- u^{\prime \prime}- \dfrac{N-1}{r}u^\prime +u = \lambda_n e^u \qquad \text{on}\ \R^+ \,,\\ 
u(0)=\gamma_n\,, \qquad u^\prime (0)=0 
\end{gathered}
\right.
\end{equation}
and $U^\ast$ satisfies \eqref{KSwhole} and \eqref{gasy} with $\lambda = \lambda_\infty$.
\end{thm}

We require the restriction $\lambda < 1/e \approx 0.37$ to guarantee 
that the nonlinearity $\lambda e^u - u$ changes sign, since otherwise due to compatibility condition, the solution $U_\lambda^*$
cannot have critical points. Also, if $\lambda < 1/e$, then 
there are  two solutions of $u = \lambda e^u$, or equivalently,  
two constant equilibria of \eqref{eq:intro_v_lambda}. 
We believe that the additional restriction on $\lambda$ in lower dimensions is technical (see Lemma \ref{liouville} below) and the result should hold without it. However, since we are interested in the asymptotes of bifurcation branches,  that is, in small $\lambda$, 
this assumption does not cause any problems below. 
Theorem \ref{KSsingintro} implies that there exists an increasing, unbounded sequence of positive real numbers 
$(R_\lambda^i )_{i=1}^ \infty$ 
depending on $N$ and $\lambda > 0$ such that
$(U^\ast_\lambda)^\prime (R_\lambda^i)=0$, that is, $U^\ast_\lambda$ satisfies Neumann boundary conditions
on $\partial B_{R_\lambda^i}$. 
Consequently,
$U_\lambda^\ast$ is a singular radial solution to \eqref{eq:intro_v_lambda} in the ball of radius $R_\lambda^i$, $i\in \N$.

Our next main result states that if the radius $R$ and any large integer $i > 1$ are fixed, 
we can choose $\lambda > 0$ such that $R^i_\lambda = R$, that is, $U^*_\lambda$ has prescribed number of intersections with 
$\bar{u}_\lambda$ on $B_R$.
Note that this result does not follow from a rescalling of the domain, since our equation is not scaling invariant. Clearly, such singular solutions have 
exactly $i$ critical points (including the one on the boundary).

\begin{thm}
\label{KSmainthm}
Assume $N\geq 3$ and let $R>0$. 
Fix any $\tilde \lambda \in (0, \lambda^*_N)$ (cf. Theorem \ref{KSsingintro}) and let 
$U^*_{\tilde{\lambda}}$ be the function constructed in 
Theorem \ref{KSsingintro}. Denote by
$(R_{\tilde \lambda}^{i})_{i \in \N}$ the increasing sequence such that  $(U^*_{\tilde{\lambda}})' (R_{\tilde \lambda}^{i}) = 0$
and let
$i^\ast$ be the 
smallest integer such that $R_{\tilde \lambda}^{ i^\ast}>R$.  
Then, for any $i\geq i^\ast$, there exists $\lambda^i>0$ such that
$$
R^i_{\lambda^i} =R.
$$
In particular,  for any $i\geq i^\ast$, there exists $\lambda^i >0$ such that the equation \eqref{eq:intro_v_lambda} admits a singular radial solution $U_{\lambda^i}^\ast$ satisfying
$$
\sharp \{r\in (0,1)| U^\ast_{\lambda^i} (r)=\overline{u}_{\lambda^i} \}=i.
$$
\end{thm}

 Once the existence of singular solutions on bounded domains is established, we turn our attention to the 
character of bifurcation branches parametrized by the value of solutions at the origin. 

First we claim that the branch $\mathcal{B}_i^+$ (see Theorem \ref{thm:bifurcation}) is bounded in $\mu$, that is, 
$\mathcal{B}_i^+ \subset (0, C_i) \times (1, \infty)$, where $C_i$ depends only on $i$. 
Indeed, by testing \eqref{eq:intro_v_lambda} with $v$ we see  that there is no positive solution if $\lambda \leq 0$ and
therefore by \eqref{ceqs} no solution of \eqref{eq:intro_u_mu} if $\mu \leq 0$. 
Next, let $u$ be a solution to \eqref{eq:intro_u_mu} such that $u(0)>1$. Then setting $\tilde{u}=u-1$, we see that
$$
-\Delta \tilde{u} +\tilde{u} = e^{\mu \tilde{u}}-1 \geq \mu \tilde{u}.
$$
Hence, the Sturm-Picone comparison theorem implies that $\tilde{u}-1$ has 
arbitrary large number of zeros if  $\mu$ is large. However, since number of zeros is constant along $\mathcal{B}_i^+$ (cf. Theorem \ref{thm:bifurcation}), the claim follows. 

However, the branch $\mathcal{B}_i^+$ (see Theorem \ref{thm:bifurcation}) is unbounded, and therefore by Theorem \ref{KSsingintro},
$\mathcal{B}_i^+$ (up to sub-sequence) converge to singular solutions. 
Next, we turn our attention to asymptotic behavior of $\mathcal{B}_i^+$.

To formulate the next result, for given $\lambda, \gamma > 0$ we denote by $(r_{\lambda,\gamma}^i)_i$ 
the increasing sequence 
satisfying $u'(r_{\lambda,\gamma}^i , \gamma) = 0$, where $u(\cdot, \gamma)$ is the unique solution to \eqref{KSeqintfin}.
Note that if $u(\cdot, \gamma)$ is non-constant, its critical points are necessarily discrete and the sequence $(r_{\lambda, \gamma}^i)_{i}$
is either finite or countable. 

The following theorem gives a strong indication that for each $i \geq 1$,  the branch $\mathcal{B}_i^+$ 
oscillates around $\lambda^i$ when $3\leq N\leq 9$. 
Below we show that the oscillations of $\mathcal{B}_i^+$ indeed take place for a generic radius. 

\begin{thm}
\label{KSoscbranch}
Fix $3\leq N\leq 9$, $R>0$, $i \geq i^*$ (see Theorem \ref{KSmainthm}), and let $\lambda^i>0$ be the positive real number given in Theorem \ref{KSmainthm}. Then, there exists a sequence of initial data $(\gamma_n)_n$ with $\gamma_n \rightarrow \infty$ and a sequence of positive integer $(j_n)_n$ such that $r_{\lambda^i ,\gamma_n}^{j_n}=R$. 
\end{thm}

Another evidence that the branch $\mathcal{B}_i^+$ oscillates around $\lambda^i$ infinitely many times if $3 \leq N \leq 9$
and finitely many times if $N > 10$ is provided by the Morse index of the singular solution. We leave open the border line case $N=10$.
Recall that 
the Morse index of $v$ satisfying \eqref{eq:intro_v_lambda}, denoted $m(v)$, in the space of radial functions 
is the number of negative eigenvalues $\alpha$ (counting multiplicities) of the eigenvalue problem
\begin{equation}
\left\{
\begin{aligned}
-\Delta \phi +\phi - \lambda e^{v}\phi &= \alpha \phi &\qquad &\text{in} \ B_R,\\ 
\partial_\nu \phi &= 0 & &\text{on} \ \partial B_R,\\
 \phi\ \text{is radially} &\text{ symmetric.} && 
\end{aligned}
\right.
\end{equation}
Recall that the Morse index of solutions remains constant along  a bifurcation 
branch unless it has a critical point in $\lambda$. Thus, 
each turning point of a bifurcation 
branch corresponds to a transition of an eigenvalue (of the linearization) across imaginary axis. 
Since the solutions $u(\cdot, \gamma) \to U^*$ as $\gamma \to \infty$, 
the Morse index of $U^*$ indicates the total number of turning points of the branch and combined with Theorem \ref{KSoscbranch},
it suggests the number of intersection points of  $\mathcal{B}_i^+$ with $\lambda^i$.

\begin{prop}
\label{KSmorse}
If $U^\ast_{\lambda^i}$ is a solution to \eqref{eq:intro_v_lambda}, then 
$m(U^\ast_{\lambda^i})<\infty$ when  $N> 10$ while $m(U^\ast_{\lambda^i})=\infty$ when  $3\leq N\leq 9$.
\end{prop}

Finally, we prove the oscillation of the branches $\mathcal{B}_i^+$ in dimension $3\leq N\leq 9$ for generic radius. 
If the scale invariance is available, then one can show that $\mathcal{B}_i^+$ can be parametrized by the value of 
the solution at the origin, and in particular there are no secondary bifurcations and singular solutions 
are non-degenerate. In our case the situation is much 
more complicated and we rely on Sard's theorem which merely yield results for generic domains.   

First, we show a generic local uniqueness result for singular solutions, which combined with Theorem \ref{KSsingintro}
yields that $\mathcal{B}_i^+$ (and any other branches) converge to discrete set of functions.  
More precisely, 
for generic $R>0$, if $(U^*_{\lambda^*})'(R) = 0$, then $(U^*_{\lambda})'(R) \neq 0$, for 
$\lambda$ close but different to $\lambda^\ast$. In other words, if we have a singular solution on $B_R$ for 
certain $\lambda^*$, then we do not have a singular solution for nearby $\lambda$'s, that is, the set $(\lambda^i)$ (see 
Theorem \ref{KSmainthm}) is discrete. 
 
\begin{thm}
\label{generic}
There exists a set $S^* \subset (0, \infty)$ of Lebesgue measure zero, such that for any radius $R \in (0, \infty) \setminus S^*$
the following holds. If $(U^*_{\lambda^*})'(R) = 0$, then there exists $\delta > 0$ such that for any $\lambda \in 
(\lambda^* - \delta, \lambda^* + \delta) \setminus \{\lambda^*\}$ one has $(U^*_{\lambda})'(R) \neq 0$.
\end{thm}

A direct consequence is the following corollary.

\begin{cor}
Let $R\in (0,\infty ) \backslash S^\ast$, where $S^\ast$ is defined in Theorem \ref{generic}. Then, there exists 
$\delta>0$ such that, for any $\lambda \in (\lambda^i - \delta ,\lambda^i +\delta)\backslash \{\lambda^i \}$, 
there is no singular solution of \eqref{eq:intro_v_lambda} satisfying \eqref{gasy}.
\end{cor}

To formulate a generic uniqueness result for regular solutions, recall $r^i_{\lambda ,\gamma}$ defined in Theorem \ref{KSoscbranch}. 
Then, for any $R\in (0, \infty) \setminus S^\ast$ any any large $\gamma$ there exists at most one $\lambda \approx \lambda^i$ such that 
$r^i_{\lambda ,\gamma} = R$.

\begin{thm}
\label{generic2}
Fix $\lambda^i$ as in Theorem \ref{KSmainthm} and let $S^*$ be the zero measure set as in  
Theorem \ref{generic}. Then for any $R \in (0, \infty) \setminus S^*$, 
 there exist $\delta > 0$ and $\Gamma > 0$ such that for each 
$\gamma \geq \Gamma$ there exists at most one $\lambda \in (\lambda^i - \delta, \lambda^i + \delta)$
such that $ r^{i}_{\lambda,\gamma} = R$. 
\end{thm}

As a direct corollary of the two previous theorems and Theorem \ref{KSoscbranch}, we obtain a quite complete picture of the bifurcation diagram in small dimension for generic radius.

\begin{cor}
Let $3\leq N\leq 9$. For $R\in S^\ast$, the branches $\mathcal{B}_i^+$ defined in \ref{thm:bifurcation} oscillate in the plane 
$(\mu , u(0))$ around the line $\mu = \mu^i$, where $\lambda^i e^{\mu^i}=\mu^i$. 
Moreover, no secondary bifurcation occurs for large $u(0)$ and there are no branches bifurcating from infinity. Furthermore, 
$\mathcal{B}_i^+$ can be parametrized by $u(0)$ for large values of $u(0)$. 
\end{cor}

Let us briefly describe the main ideas of proofs. 
We often use the change of variables
\begin{equation*}
u(r)=\eta(\zeta)+2\zeta,
\end{equation*}
where
$$
 r=\sqrt{\dfrac{2(N-2)}{\lambda }}e^{-\zeta}
$$
which transforms a radial solution $u$ of \eqref{eq:intro_v_lambda} to $\eta$ satisfying
\begin{equation}\label{inez}
\left\{
\begin{aligned}
\eta^{\prime \prime}  - (N-2) \eta^\prime +2(N-2) \eta &= m^2 e^{-2\eta}(\eta + 2 \zeta) - 2(N-2) (e^\eta -1-\eta)
\qquad \eta\in \R,\\
 \displaystyle\lim_{\zeta\rightarrow \infty} \eta(\zeta)&=0.
\end{aligned}
\right.
\end{equation}
Note that the zero order term $u$ makes \eqref{inez} non-autonomous and as such we cannot 
directly use techniques from dynamical systems. 
However, to gain a better intuition assume that the term $m^2 e^{-2\eta}(\eta + 2 \zeta)$, which is exponentially 
small at infinity, is missing. In that case, we are searching for solutions converging along stable manifold to $0$.
A standard linear analysis yields that $0$ is an unstable focus if $3 \leq N \leq 9$ and unstable node if $N \geq 10$ and as 
such there is no stable manifold. Thus, if $e^{-2\zeta}(\eta + 2\zeta)$ is missing, then $\eta \equiv 0$ is the only 
solution of \eqref{inez}. This reasoning suggests that solutions of \eqref{inez} are unique and exponentially close 
to 0 at least for large $\zeta$. 
The uniqueness yields that solutions of \eqref{inez} are very unstable and are presumably hard to analyze by direct numerical
and analytical methods. 
Thus, to prove the existence and uniqueness of solutions to \eqref{inez} we incorporate the condition at infinity 
into the choice of functional spaces and use the Banach fixed 
point theorem. 

To analyze the oscillations, we need to understand the behavior of $U^\ast_\lambda$ for large $r$. Since 
\eqref{eq:intro_v_lambda} admits a Lyapunov functional, intuition (modulo non-autonomous term $\frac{1}{r}u'$ which is 
small for large $r$)
yields that the function $U^\ast_\lambda$ converges as $r \to \infty$ to an equilibrium of the \eqref{eq:intro_v_lambda}
viewed as an initial value problem, that is, 
to the values $\bar{u}_\lambda$ or $\underline{u}_\lambda$ (see \eqref{ceqs}). Again by ignoring 
 the term 
$u'/r$, we can analyze the character of  equilibria and obtain that $\underline{u}_\lambda$ is a saddle and  
$\bar{u}_\lambda$ is a center. Hence, the former does not allow for oscillatory solutions, whereas the latter does. Therefore, 
an important ingredient of the proof is to show that the singular solution of \eqref{KSwhole} 
does not converge to $\underline{u}_\lambda$ as $r\rightarrow \infty$, see Proposition \ref{liouville} which is in fact a Pohozaev type
identity. This is the only result where we need our technical upper bound $\lambda \leq \lambda^*$ in lower dimensions 
(cf. Theorem \ref{KSsingintro}). 
The final argument is based on Sturm-Piccone oscillation theorem and careful estimates of singular solutions. Note that 
similar ideas were used in \cite{nitechnote}.

The proof of $u(\cdot,\gamma)\rightarrow U^\ast$ as $\gamma \rightarrow \infty$ is partly motivated by \cite{miya2} and 
crucially depends on the uniqueness of the singular solution $U^\ast$. Then, it suffices to prove that $u(\cdot,\gamma)$
converges to a function that satisfies both the equation \eqref{KSwhole} and asymptotics at the origin \eqref{gasy}. Since, 
$u(\cdot,\gamma)$ and $U^\ast$ satisfy the same equation \eqref{KSwhole}, the convergence of $u(\cdot,\gamma)$ to a solution of 
\eqref{KSwhole} follows from a priori estimates and standard regularity theory. The asymptotics at the origin is of a different flavor 
and requires careful estimates in transformed variables. 

The proof of Theorem \ref{KSoscbranch} uses an observation that for fixed $\lambda$ 
and large $u$, the term $u$ is negligible compared to $e^u$, and therefore if $\gamma$ is large, then 
close to the origin we can neglect the zero order term which was responsible for the breaking of scaling. 
Hence, close to the origin $u$ can be approximated by the solution of scale invariant problem
\begin{equation}\label{ruva}
\left\{ 
\begin{gathered} 
\Delta \bar{u} + \lambda e^{\bar{u}}=0\qquad \text{on} \ (0,\infty) \,, \\
\bar{u}(0,\alpha)=\alpha,\qquad \bar{u}_\rho (0,\alpha)=0\,.
\end{gathered}
\right.
\end{equation}
The same reasoning yields that singular solutions of \eqref{eq:intro_v_lambda} can be approximated near the origin
by the singular solution of \eqref{ruva} which is given by
\begin{equation*}
\bar{u}^\ast (r)= -2\ln r +\ln \dfrac{2(N-2)}{\lambda}.
\end{equation*}
Using 
the classical arguments of Joseph and Lundgren \cite{josephlundgren} and scale invariance of \eqref{ruva} 
we conclude that if $3 \leq N \leq 9$ and $\alpha$ being sufficiently large,  the solution 
$u(\cdot, \alpha)$ of \eqref{ruva} intersects arbitrarily many times $\bar{u}^*$ in a small neighborhood of the origin.  
Using precise estimates we can indeed verify this intuition and 
conclude that the solution $u$ of \eqref{eq:intro_v_lambda} with $u(0) = \gamma$
intersects arbitrarily many times the singular solution $U^*$ in a small neighborhood of the origin. 
The rest of the proof of Theorem \ref{KSoscbranch} follows from
zero number arguments.

Theorem \ref{KSmainthm} is a consequence of the continuity of the function $\lambda\rightarrow M_\lambda^i$ for all $i\in \N$ and the fact that, for any $i\in \N$,
\begin{equation}
\label{LNTmaine1intro}
M_\lambda^i \to 0^+,\ \textrm{ as }\ \lambda\to 0^+.
\end{equation}
Although this idea is rather elementary its proof poses the main technical
challenge of the paper. In order to prove \eqref{LNTmaine1intro} we not only need more precise asymptotics of 
$U^*_\lambda$ at the origin, but we require estimates on the \textit{length of the interval} where the asymptotics are valid. 
In fact, we prove estimates up till $r_\lambda$, the first intersection point of $U^*_\lambda$ with $\bar{u}_\lambda$. The cornerstone
of the proofs is an observation that the higher order correction of $U^*_\lambda$ for small $r$ is negative. Once the 
first intersection with $\bar{u}_\lambda$ is established, we obtain an estimate on $(U^*_\lambda)'(r_\lambda)$ and finish
the proof using careful estimates and Sturm-Piccone theorem. We remark that direct estimates up till the first critical 
point of $U^*_\lambda$, that is, on $R_\lambda^1$,  seem beyond reach. The continuity of  
$\lambda\rightarrow M_\lambda^i$ is primarily based on the uniqueness of $U^*_\lambda$.

The bounds on the Morse index stated in Proposition \ref{KSmorse} are based on the asymptotic behavior of 
$U^\ast_\lambda$ when $r\rightarrow 0$ combined with Hardy's inequality. 
The proof of Theorem \ref{generic} follows from the fact that the function $\lambda \rightarrow R_\lambda^i$ is Lipschitz which allows us to use the Sard's theorem. Lipschitz continuity in turn follows from precise estimates on the modulus of continuity of the function 
$\lambda \mapsto U^*_\lambda$. 
The main observation in the proof of Theorem \ref{generic2} is the fact that the function $\lambda \rightarrow r^i_{\lambda ,\gamma}$ is bounded in $C^2 (I)$, for some compact interval $I \subset (0,\infty )$, by a constant not depending on $\gamma$, which in
turn follows from precise estimates on the rate of convergence of regular solutions to singular ones. 

The paper is organized as follows. In Section \ref{secexis}, we prove the first part of Theorem \ref{KSsingintro}, 
namely, we establish the existence of $U_\lambda^\ast$ and prove oscillations 
around $\overline{u}_\lambda$. We finish the proof of Theorem \ref{KSsingintro} in Section \ref{secconv} by showing the convergence
of $u_\lambda (r,\gamma)$ to $U_\lambda^\ast (r)$ as 
 $\gamma \rightarrow \infty$. Section \ref{sec4} is dedicated to the proofs of Theorem \ref{KSoscbranch} 
 and Proposition \ref{KSmorse} and Theorem \ref{KSmainthm} is proved in Section \ref{sec5}. Finally, generic results, 
 Theorems \ref{generic} and \ref{generic2} are proved in Section \ref{sec:gen}.
Let us mention that we expect the same results to hold for the Lin-Ni-Takagi equation or all radii, which 
will be the subject of a forthcoming work.

\section{Construction of the positive radial singular solution in the whole space.}
\label{secexis}
Fix $N \geq 3$, $\lambda > 0$ and consider the equation
 \begin{equation}
\label{sectKS}
\left\{
\begin{aligned}
- u'' - \frac{N - 1}{r}u' + u &=\lambda e^u,&\qquad  &\textrm{in} \ (0, \infty),\\
u &> 0,& &\textrm{in} \ (0, \infty) \,,
\end{aligned}
\right.
\end{equation}
where $u$ depends on the radial variable $r$ and the derivatives are with respect to $r$. 
The main goal of this section is the proof of the existence and uniqueness of solution of \eqref{sectKS} 
with 
\begin{equation}\label{ubbc}
u(r) = -2 \ln r + \ln \frac{2(N-2)}{\lambda} + o(1) \qquad \textrm{as } r\to 0^+ \, ,
\end{equation} 
where we denote by $o(1)$ the class of functions $f$ such that $\lim_{r\rightarrow 0^+} f(r)=0$. We use the following change of variables
\begin{equation}
\label{defw}
u(r)=\eta(\zeta) + 2\zeta,
\end{equation}
where
\begin{equation}\label{nrv}
 r=\sqrt{\dfrac{2(N-2)}{\lambda }}e^{-\zeta}.
\end{equation}
To simplify notation, we denote $m=\sqrt{\dfrac{2(N-2)}{\lambda }} $ and $\zeta = \ln \dfrac{m}{r}$. A direct computation shows that
\begin{equation}\label{fdrv}
\dfrac{ du}{dr} = - \dfrac{1}{r} \dfrac{ d \eta}{d\zeta} -\dfrac{2}{r} ,
\end{equation}
and
$$
\dfrac{ d^2 u}{dr^2} = \dfrac{1}{r^2} \dfrac{ d^2 \eta}{d\zeta^2} + \dfrac{1}{r^2}\dfrac{d \eta}{d\zeta} +\dfrac{2}{r^2}.
$$
In the following, if $f:\R\rightarrow \R$ depends only on one variable $\rho$, usually $r$ or $\zeta$, 
we denote $f ' = \frac{d f}{d\rho}$, and analogously for higher order derivatives. 
Then,  \eqref{sectKS} is equivalent to 
\begin{align}\label{eq:rdu}
0&= - u^{\prime \prime}-\dfrac{N-1}{r}u^\prime +u - \lambda e^u\\
&=   \dfrac{1}{r^2} \left[-\eta^{\prime \prime}  +(N-2) \eta^\prime  + m^2 e^{-2\zeta} (\eta + 2\zeta ) -2(N-2)(e^\eta-1) \right]\,,
\end{align}
and consequently
\begin{equation}\label{eq:fze}
\eta^{\prime \prime}(\zeta ) -(N-2)\eta^\prime (\zeta)+2(N-2) \eta (\zeta ) =g(\zeta),
\end{equation}
where
\begin{equation}\label{dfge}
g(\zeta)= m^2 e^{-2\zeta}(\eta (\zeta )+2\zeta ) - 2(N-2) (e^{\eta (\zeta )} -1-\eta (\zeta ))
\end{equation}
We also set
\begin{equation}\label{dfph}
\phi(\eta) = - 2(N-2) (e^{\eta} -1-\eta) \,.
\end{equation}
The blow up rate \eqref{ubbc} is equivalent to 
\begin{equation}
\lim_{\zeta \to \infty} \eta(\zeta) = 0 \,.
\end{equation}
For any $N \geq 3$ denote 
\begin{equation}\label{dfab}
\alpha = N - 2, \qquad \beta = \sqrt{\dfrac{(N-2)|N-10|}{4}}, \quad 
\end{equation} 
and let $G_N$ be the Green's function for the left hand side of \eqref{eq:fze} defined by
\begin{equation}\label{dfgf}
G_N(z) := 
\begin{cases}
\frac{1}{\beta} e^{-\frac{\alpha}{2}z} \sin(\beta z) & 3 \leq N \leq 9, \\
e^{-\frac{\alpha}{2}z} z & N = 10, \\
\frac{1}{\beta} e^{-\frac{\alpha}{2}z} \sinh(\beta z) & N > 10 
\end{cases}
\quad \textrm{for } z \geq 0, \qquad G_N (z) = 0 \qquad \textrm{for } z < 0 \,.
\end{equation}
Observe that $G_N \in L^1(\mathbb{R}) \cap L^\infty(\mathbb{R})$ for any $N \geq 3$.  
Then, \eqref{eq:fze} is equivalent to  
\begin{equation}\label{intformzeta}
\eta (\sigma)=  \int_{\R} G_N(\tau - \sigma) g(\tau)d\tau \,.
\end{equation}
Thus, finding solution of \eqref{sectKS} satisfying \eqref{ubbc} reduces to finding a solution of \eqref{intformzeta}.

\begin{prop}
\label{uniq}
Let $m>2\sqrt{2(N-2)}$. The equation \eqref{eq:fze} admits unique solution on $(-\infty, \infty)$ satisfying 
\begin{equation}
\label{KSuniqeq}
 \displaystyle\lim_{\zeta \rightarrow \infty} \eta(\zeta) = 0. 
\end{equation}
 This solution is also unique on any interval $(\zeta_0, \infty)$, $\zeta_0 \in \R$.
\end{prop}

\begin{rmq}
Proposition \ref{uniq} establishes existence and uniqueness of solution $U^*$ asserted in Theorem \ref{KSsingintro}. 
\end{rmq}

\begin{proof}[Proof of Proposition \ref{uniq}]
First, we construct a local solution by using the contraction mapping theorem
on the Banach space 
$X=\{\eta \in C^0 ([\zeta_0;\infty)); |\eta|_{\infty}<\infty\}$,  where $\zeta_0$ is determined below and $C^0 ([\zeta_0, \infty))$ 
is the space of continuous function on $[\zeta_0, \infty)$ that decay at infinity, equipped with the supremum norm. 
Also, for any $\bar{r} \geq 0$ denote 
$\mathcal{B}_{\bar r} =\{\eta \in X; |\eta|_{\infty}< \bar{r}\} $ and let $g$ be as in \eqref{dfge}. 
To avoid confusion, we explicitly indicate the dependence of $g$ on $\eta$.

Let $G_N$ be defined by \eqref{dfgf}. For any $\eta \in \mathcal{B}_{\bar r}$ and any $\zeta \geq \zeta_0$, denote 
$$
F(\eta)(\zeta) = \int_{\R} G_N(\sigma - \zeta) g(\eta ,\sigma)d\sigma = \int_{0}^\infty G_N(\sigma) g(\eta ,\sigma + \zeta)d\sigma\,.
$$ 
Note that the integrals are well defined since $G_N \in L^1$ and $G_N (z) = 0$ for $z \leq 0$.  Since $\eta \in X$, we have that $\eta (\zeta) \to 0$
as $\zeta \to \infty$, and therefore $|g(\eta, \zeta)| \to 0$ as $\zeta \to \infty$. Hence, since $G_N \in L^1$ 
\begin{equation}
|F(\eta)(\zeta)| \leq C_N \sup_{\sigma \geq 0} |g(\eta ,\sigma + \zeta)|  \to 0 \qquad \textrm{as } \zeta \to \infty 
\end{equation}
and in particular $F : X \to X$. 

Next, we show that $F$ is a contraction on $\mathcal{B}_{\bar r}$. Indeed, for any $\varepsilon > 0$ there is $\bar{r} > 0$ and $\zeta_0 > 0$ such that for every $\eta_1, \eta_2 \in \mathcal{B}_{\bar{r}}$ and $\zeta \geq \zeta_0$ 
one has
\begin{align}
 |g(\eta_1,\sigma)-g(\eta_2,\sigma)| &\leq 
(m^2 e^{-2\sigma}  |\eta_1(\sigma) - \eta_2(\sigma)| + 2(N - 2)|e^{\eta_1(\sigma)} - e^{\eta_2(\sigma)}- \eta_1 (\sigma )+\eta_2 (\sigma ) |  
\\
&\leq \varepsilon \|\eta_1 -\eta_2\|_\infty \,,
 \end{align}
 where in the last step we used the mean value theorem for the function $m(x) = e^x - x$ and the fact that $|m'(x)| = |e^x - 1|$ is small if $x$ is small,
 that is, if $|\eta_i (\sigma)| \leq \bar{r} \ll 1$ for $i \in \{1, 2\}$ . 
Then, since $G_N \in L^1$
\begin{align}
\left\|F(\eta_1)-F(\eta_2)\right\|_{L^\infty ((\zeta_0, \infty))} 
\leq \varepsilon \left\|\eta_1 -\eta_2\right\|_{L^\infty ((\zeta_0, \infty))} \|G_N\|_{L^1} = 
 C_N \varepsilon \left\|\eta_1 -\eta_2\right\|_{L^\infty((\zeta_0, \infty))} \, .
\end{align}
Fix $\varepsilon > 0$ such that $C_N \varepsilon < \frac{1}{2}$, which in turn fixes $\bar{r}$.  

Finally, we show that $F$ maps $\mathcal{B}_{\bar{r}}$ into itself. 
By increasing $\zeta_0$ if necessary, we can assume that $\zeta e^{-2\zeta} < \varepsilon_0 \bar{r}$ for any $\zeta \geq \zeta_0$, 
where $0 < \varepsilon_0 < \frac{1}{C_N m^2}$. Then, for any $\zeta \geq \zeta_0$,
$$ 
\left\|F(0)\right\|_{L^\infty(\zeta_0, \infty)} = \sup_{\zeta \geq \zeta_0} 
 2m^2 \left| \int_0^\infty G_N(\sigma) e^{-2(\sigma  +\zeta)} (\sigma + \zeta) d\sigma \right|
< \varepsilon_0 C_N m^2\bar{r}  < \frac{1}{2} \bar{r}\,.
$$
Thus for any $\eta \in \mathcal{B}_{\bar{r}}$
one has 
$$
\left\|F(\eta)\right\|_{L^\infty(\zeta_0, \infty)}  \leq \left\|F(\eta) - F(0)\right\|_{L^\infty(\zeta_0, \infty)} + \left\|F(0)\right\|_{L^\infty(\zeta_0, \infty)} 
< \frac{1}{2} \|\eta\|_{L^\infty(\zeta_0, \infty)} + \frac{1}{2} \bar{r} \leq \bar{r} \,,
$$
and so $F$ is a contraction on $\mathcal{B}_{\bar{r}}$. The existence and uniqueness of solutions on $(\zeta_0, \infty)$
follows from the Banach fixed point theorem. To prove the uniqueness in $X$ suppose that there are two solutions $\eta_1$ and 
$\eta_2$. Fix $\bar{r}$ as above and by \eqref{KSuniqeq} we can choose $\zeta_0$ sufficiently large such that 
$\eta_1, \eta_2 \in \mathcal{B}_{\bar{r}}$. By the already proved uniqueness we obtain that $\eta_1 = \eta_2$ on $(\zeta_0, \infty)$. The fact that $\eta_1 \equiv \eta_2$ follows from the uniqueness of the initial value problems.  

Let us prove that the solution can be extended to the whole real line. We proceed by showing that a solution $u$ of
\eqref{eq:rdu} defined on the interval $(0, r_0)$ can be extended to the interval $(0, \infty)$. Indeed, let $(0, R_0)$
 be the maximal existence interval of the solution and assume $R_0 < \infty$. 
Since the nonlinearity is Lipschitz it suffices to show that $u$ is bounded on the interval $I_0 = (R_0/2, R_0)$.   
Next, observe that the functional 
\begin{equation}\label{dflf}
V(r)=\dfrac{(u'(r))^2 - u^2 (r) }{2}+\lambda e^{u(r)}
\end{equation}
is a Lyapunov functional for the flow, that is, $r \mapsto V(r)$ is decreasing on $r \in (0, R_0)$. Hence,   
\begin{equation}\label{lypf}
V(r) \leq V(R_0/2) = C^* \qquad \textrm{for any } r \in  I_0 \,,
\end{equation}
that is, $V$ is bounded from above on $I_0$. To prove that $u$ is bounded,
note that   \eqref{lypf} yields $(u'(r))^2 - u^2 (r) \leq C^* $, and therefore 
\begin{equation}
(u^2)' = 2uu'  \leq u^2 + (u')^2 \leq C^* + 2 u^2  \,.
\end{equation}
The Gronwall inequality yields that $u^2(r) \leq Ce^{2r}$ for $r \in I_0$, where $C$ depends on $C^*$, $R_0$, and $u(R_0/2)$. Thus, $u$ is bounded on $I_0$, and therefore can be continued beyond $R_0$,  a contradiction. 
 \end{proof}

Next we obtain more precise asymptotics on $w$ at infinity, which in turn transforms into more precise asymptotics of $u$ at the origin. 

\begin{lem}
\label{behaorigin}
If $\eta$ is a solution of \eqref{eq:fze}, \eqref{KSuniqeq}, then for any $\delta>0$,
$$
\lim_{\zeta \to \infty} e^{(2-\delta) \zeta} \eta (\zeta)  =  \lim_{\zeta \to \infty} e^{(2-\delta) \zeta} \eta' (\zeta) = 0 \,.
$$
\end{lem}

\begin{proof}
By applying  Young convolution inequality to \eqref{intformzeta}, we have 
\begin{align*}
\int_\zeta^\infty |\eta (\sigma)|d\sigma \leq  \|G_N\|_{L^1} \int_\zeta^\infty |g(\sigma )|d\sigma = 
C_N \int_\zeta^\infty |g(\sigma )|d\sigma \,.
\end{align*}
Since, for every $\varepsilon > 0$,
one has $2(N - 2)|e^a - 1 - a| \leq \varepsilon |a|$  for any sufficiently small $|a|$,  and since $\eta (\sigma )\to 0$ as $\sigma \to \infty$,  we deduce that for 
any $\delta>0$, there exists large  $\zeta_0$ such that, for any $\zeta \geq \zeta_0$,
\begin{equation}
\label{estg}
|g(\zeta)|\leq \varepsilon |\eta (\zeta)| + 2 m^2 \zeta e^{-2\zeta} + m^2 e^{-2\zeta} |\eta(\zeta)| \leq 2 \varepsilon (e^{-2(1-\delta/2)\zeta} + |\eta(\zeta)|).
\end{equation}
This implies that, for $\varepsilon = \frac{1}{4}$,  any sufficiently small $\varepsilon, \delta > 0$ and $\zeta \geq \zeta_0$ 
$$
 \int_\zeta^\infty |\eta(\sigma )|d\sigma \leq C e^{-2(1-\delta/2)\zeta} \,,
$$
where $C$ depends on $\delta$ and $N$.
Substituting this estimate and \eqref{estg}  with $\varepsilon = \frac{1}{4}$ to \eqref{intformzeta}, we obtain that 
\begin{align}
\label{esteta}
|\eta (\zeta)|&\leq \frac{1}{2} \int_\sigma^\infty |G_N(\tau - \sigma)| (e^{-2(1-\delta)\sigma}+|\eta (|\sigma | ) d\tau  \leq C e^{-2(1-\delta/2)\zeta}
\end{align}
and the first assertion follows. 

Finally, since 
\begin{equation}
\eta'(\sigma) =  \int_{\R} G_N'(\tau - \sigma) g(\tau)d\tau
\end{equation}
and $G_N' \in L^1$, we can proceed as above by replacing $G_N$ by $G_N'$ and conclude the proof.
\end{proof}

Next, we show that $U^\ast \in H^1_{loc}(\R^N)$ where $U^\ast$ is defined in Theorem \ref{KSsingintro}.

\begin{lem}
\label{decaderu}
If $U^\ast$ is as in Theorem \ref{KSsingintro}, then
$$
\lim_{r \to 0} (U^\ast )^\prime (r) + \dfrac{2}{r} = 0.
$$
Moreover,  $U^\ast \in H^1 (B_{R_1})$ for any $R_1 > 0$.
\end{lem}

\begin{proof}
By Proposition \ref{uniq},  $U^*$ exists on $(0, \infty)$. 
Relation \eqref{fdrv} implies that
\begin{align*}
(U^\ast )^\prime (r)&= 
- \dfrac{1}{r} \eta ' \left(\zeta \right) - \dfrac{2}{r} = - \frac{1}{m} e^\zeta \eta ' \left(\zeta \right) - \dfrac{2}{r}
\end{align*}
and the first assertion follows from Lemma \ref{behaorigin}. 
Next, recall that  $U^\ast (r) = \eta(\zeta) - 2 \ln \frac{r}{m}$. So, using Lemma \ref{behaorigin}, we deduce that
\begin{align*}
\left\|U^\ast \right\|_{H^1 (B_{R_1})}^2 &= \omega_N \int_0^{R_1} (|(U^\ast)' |^2 +| U^\ast |^2 ) r^{N-1} dr\\
&\leq C\int_0^{R_1} (r^{-2}+ (\ln r)^2 +1)r^{N-1}dr<\infty 
\end{align*}
for $N \geq 3$. This establishes the lemma.
\end{proof}

Next, we focus on the behavior of $U^*$ for large $r$. As a preliminary we prove the following lemma which is a Pohozaev-type
identity.

\begin{lem}\label{liouville}
Fix $N \geq 3$ and $\lambda \in (0, \lambda_N^*)$, where $\lambda_N^*$ is as in Theorem 
\ref{KSsingintro}.  If $U^*$ is the unique solution of \eqref{KSwhole}, then $U^* > \underline{u}_\lambda$ and 
\begin{equation}
\liminf_{r \to \infty} U^*(r) > \underline{u}_\lambda \,,
\end{equation}
where $\underline{u}_\lambda < 1$ is the smaller solution of $u = \lambda e^u$. 
\end{lem}

\begin{proof}
If $v = U^\ast - \underline{u}_\lambda$, 
then $v$ satisfies 
\begin{equation}\label{eqliouville}
-\Delta v +v= \underline{u}_\lambda (e^v -1)
\end{equation}
with 
\begin{align}
v(r) &= - 2\ln r + \ln \frac{2(N-2)}{\lambda} - \underline{u}_\lambda + O(1) = 
- 2\ln r + \ln \frac{2(N-2)}{\lambda e^{\underline{u}_\lambda}} + O(1) \\
&=
- 2\ln r + \ln \frac{2(N-2)}{\underline{u}_\lambda} + O(1) \,.
\end{align}
For a contradiction, assume that either there exists the smallest $R_0$ such that $v(R_0) = 0$, 
or $v > 0$ and $\liminf_{r \to \infty} v(r) = 0$. In the latter case
we set $R_0 = \infty$. Denote $\varepsilon_0 = 1 - \underline{u}_\lambda \in (0, 1)$.

We claim that $R_0 = \infty$ implies $\lim_{r \to \infty} v(r) = 0$. Indeed, if not then there exist
$v_0 > 0$ and a sequence $r_n \to \infty$
as $n \to \infty$ such that $v(r_n) \geq v_0 > 0$. Since $\liminf_{r \to \infty} v(r) = 0$, by the mean value theorem,  
there is a local minimizer $r^*$ of $v$. In particular,  $v'(r^*) = 0$
and thus $(U^*)'(r^*) = 0$. Since the Lyapunov functional $V$ defined by \eqref{dflf} is decreasing, we obtain that
$V(r^*) > V(r)$ for any $r > r^*$. This implies that there is no $r > r^*$ such that $U^*(r) = U^*(r^*)$ and since $r^*$
is a local minimum $U^*(r) \geq U^*(r^*) > 0$ for any $r \geq r^*$. This contradicts $\liminf_{r\to \infty} U^*(r) = 0$, 
and the claim follows. 
 
If $R_0 = \infty$, then since $v \searrow 0$, 
we can  fix $R > 0$
such that $|e^{v(r)} - 1| \leq (1 + \varepsilon_0)|v(r)|$ for each $ r \geq R$.  Consequently, 
\begin{equation}
-\Delta v + v \leq  \underline{u}_\lambda (1 + \varepsilon_0) v = (1 - \varepsilon_0^2) v \qquad \textrm{in } \R^N \setminus B_{R}(0)\,.
\end{equation} 
 Define $\psi (r)= C_1 e^{-\varepsilon_0/2 (r-R)}$, for some $C_1>0$ specified below. It is easy to see after increasing $R$ if necessary, that we have 
$$
-\Delta \psi +  \psi =   
\left(1 - \frac{\varepsilon_0^2}{4} \right) \psi - \varepsilon_0 \frac{(N-1)}{2r} \psi \geq \left(1 - \varepsilon_0^2 \right) \psi 
\qquad \textrm{in }\ \R^N \backslash B_R (0).
$$
Fix $R$ and choose $C_1$ such that $C_1 > v(R)$. Consequently
$\psi (R) -v(R) \leq 0$ and $\lim_{r\rightarrow \infty} (v(r) - \psi (r))=0$. 
Then, a comparison principle yields $v(r)\leq \psi (r)$, for all $r\geq R$. Also, elliptic regularity theory implies that 
 $v^\prime$ decays exponentially at infinity.

Fix any $R \in (0, R_0)$
Multiplying \eqref{eqliouville} by $r^N v^\prime$ and integrating, we find, for any $0 < \rho < R$,
\begin{align*}
\dfrac{N-2}{2}&\int_{\rho}^R |v'|^2 r^{N-1} dr + \left[\dfrac{r^N (v^\prime (r))^2}{2} \right]_\rho^R - \left[\dfrac{r^N v^2}{2}\right]_\rho^R+ 
\dfrac{N}{2} \int_{\rho}^R v^2 r^{N-1} dr\\
&+\underline{u}_\lambda \left[r^N (e^v -1-v) \right]_\rho^R
 = N \underline{u}_\lambda \int_{\rho}^R (e^v-1-v )r^{N-1}dr.
\end{align*}
On the other hand, multiplying \eqref{eqliouville} by $v r^{N-1}$ and integrating, we have
$$
\int_{\rho}^R |v'|^2 r^{N-1}dr - [r^{N-1}v v^\prime]_\rho^R + \int_{\rho}^R v^2 r^{N-1} dr = \underline{u}_\lambda \int_{\rho}^R v( e^v-1) r^{N-1}dr.
$$
Since for small $\rho$ one has $|v(\rho)| \leq C |\ln \rho|$ by assumption,  and 

\begin{equation}
|v'(\rho)| = |(U^*)'(\rho)| \leq \frac{4}{\rho},
\end{equation}
by Lemma \ref{decaderu}, we have that the lower boundary terms converge to $0$ as $\rho \to 0$. Also, 
if $R_0 = \infty$
since $v(R)$ and $v'(R)$ decay exponentially as $R \to \infty$, the upper boundary terms
decay to $0$ as $R \to \infty$. If $R_0 < \infty$, one has $v(R_0) = 0$ and clearly $R_0^N (v'(R_0))^2 \geq 0$. 
This implies that
\begin{equation}
\label{KSpoho}
\int_0^{R}v^2 r^{N-1} dr + o_R(1) \leq \underline{u}_\lambda \left( N \int_0^{R} (e^v -1-v) r^{N-1}dr -\dfrac{N-2}{2}\int_0^{R} v(e^v -1) r^{N-1}dr\right).
\end{equation}
where $o_{R}(1)\rightarrow 0$ as $R\rightarrow \infty$ if $R_0 = \infty$ and $o_R(1) = 0$ if $R_0 < \infty$.
Let us denote 
$$
f(x)= x^2 -\underline{u}_\lambda \left( N  (e^x -1-x) -\dfrac{N-2}{2} x(e^x -1) \right) \,.
$$
We will obtain a contradiction to \eqref{KSpoho} for sufficiently large $R$ if we prove that $f(x) > 0$ for any $x > 0$. 
Since  $f(0)=f'(0)=0$, it suffices to show that $f''(x)>0$ for any $x >  0$. 
A simple computation shows that $f''(x)= 2 -\underline{u}_\lambda \left( 2 e^x -\dfrac{N-2}{2} xe^x  \right) $ and that 
$$
\min_{x \geq 0} f'' (x)= 
\begin{cases}f'' \left( \dfrac{6-N}{N-2}\right),& \text{ if }N < 6,\\  
f''(0) = 2 (1 - \underline{u}_\lambda), & \text{ if }N \geq 6.
\end{cases}
$$
Since by definition $\underline{u}_{\lambda}<1$, we have $f'' (x) > 0$
for  $N\geq 6$, a contradiction.  Also, we obtain a contradiction if  
$$
\underline{u}_\lambda <\dfrac{4}{N-2} e^{-\frac{6-N}{N-2}}\approx 
\begin{cases}
0.20, &\textrm{if } N=3,\\
0.74, &\textrm{if } N=4,\\ 
0.96, &\textrm{if }\ N=5   .
\end{cases}
$$
One can check that the previous values of  $\underline{u}_\lambda < 0.74$ (resp. $\underline{u}_\lambda < 0. 96$) corresponds to 
\begin{equation}
\lambda = 
\begin{cases}
0.16, &\textrm{if } N=3,\\
0.35, &\textrm{if } N=4,\\ 
0.36, &\textrm{if }\ N=5 \,,
\end{cases} 
\end{equation}
that is, $\lambda < \lambda_N^*$.
\end{proof}

\begin{rmq} \label{rmkpoho}
Under the assumptions of Lemma \ref{liouville} one has $U^* > 0$. 
\end{rmq}

Next, we prove that $U^\ast$ oscillates around $\overline{u}_\lambda$ as claimed in Theorem \ref{KSsingintro}.

\begin{lem}
\label{corbehazero}
Let $N\geq 3$ and suppose that $\lambda \in (0, \lambda^*_N)$, where $\lambda^*_N$ is as in the statement of 
Theorem \ref{KSsingintro}. If $U^*$ is as in Theorem \ref{KSsingintro} (for the existence and uniqueness see Proposition
\ref{uniq}), then there exists a sequence $0<R_\lambda^1 <\ldots <R_\lambda^k< \ldots \rightarrow \infty$  such that 
$U^\ast (R_\lambda^k) = \bar{u}_\lambda$. In particular, there is a sequence $(R_\lambda^{*k})$ such that 
$(U^\ast)'(R_\lambda^{*k}) = 0$.
\end{lem}

\begin{proof}
By Lemma \ref{liouville} one has $M := \inf U^* > \underline{u}_\lambda$. 
If we denote
$w(r)=r^{\frac{N-1}{2}} (U^*(r) - \bar{u}_\lambda)$, then standard calculations yield that $w$ satisfies
$$
w^{\prime \prime}= \left(\dfrac{U^* - \lambda e^{U^*}}{U^* - \bar{u}_\lambda } + \dfrac{(N-1)(N-3)}{4r^2}  \right)w =: m(r)w.
$$
Set
\begin{equation}
F(x) = \frac{x - \lambda e^x}{x - \bar{u}_\lambda} \quad x \neq \bar{u}_\lambda, \qquad  F(\bar{u}_\lambda) = 1- \bar{u}_\lambda\, .
\end{equation}
It is easy to see that $F$ is continuous and $F \to -\infty$ as $x \to \infty$.  Furthermore, the numerator is positive 
if and only if $x \in (\underline{u}_\lambda, \bar{u}_\lambda)$, 
whereas the denominator is positive if and only if $x > \bar{u}_\lambda$. Thus, $F < 0$
on $(\underline{u}_\lambda, \infty)$, and consequently $F \leq - 2\varepsilon_1 < 0$ on $[M, \infty)$. 
Choose $R_2$ large such that 
\begin{equation}
\dfrac{(N-1)(N-3)}{4r^2} < \varepsilon_1
\end{equation}
and we obtain $m(r) \leq - \varepsilon_1$ for $r \geq R_2$. By the Sturm-Picone comparison theorem we obtain that $w$ has infinitely many zeros on $(R, \infty)$, 
which in particular implies that $U^*$ intersects $\bar{u}_\lambda$ infinitely many times.  
\end{proof}

\section{Convergence to the singular solution.}
\label{secconv}
  In this section, we finish the proof of Theorem \ref{KSsingintro}, that is,  for any fixed $\lambda > 0$ we show that the solution 
  $u_n$ of \eqref{KSeqintfin} converges to the solution $U^\ast$ of \eqref{KSwhole}  in $C^1_{loc}(0,\infty)$ as $n \rightarrow \infty$. Although, the framework
originates from \cite{miya2}, our setting is different due to breaking of scaling, and dependence of $\lambda$ on $n$. For clarity of notation, we often drop the subscript $n$ of functions 
if the dependence is clear from the context. 

If 
$\hat{u}_n(\rho) = u_n(r,\gamma_n)-\gamma_n$ with $\rho=e^{\frac{\gamma_n}{2}}r$, then $\hat{u}(\cdot, \gamma)$ satisfies
\begin{equation}\label{defuhat}
\left\{
\begin{gathered} 
\hat{u}'' + \frac{N - 1}{r} \hat{u}' + \lambda_n e^{\hat{u}}- e^{-\gamma_n} (\hat{u} +\gamma_n) = 0 \qquad 
\text{ in } (0,\infty) \,, \\
\hat{u}(0)=\hat{u} ' (0)=0.
\end{gathered}
\right.
\end{equation}
Next, let $\bar{u}(r,\bar{\gamma})$ be the unique radial solution of
\begin{equation}
\label{defubar}
\left\{
\begin{gathered} 
\bar{u}'' + \frac{N - 1}{r} \bar{u}' + \lambda_\infty e^{\bar{u}}=0\qquad \text{ in } (0,\infty)\,,\\
\bar{u}(0)=0,\quad \bar{u} ' (0)=0\,.
\end{gathered}
\right.
\end{equation}
The existence of global solutions of \eqref{defuhat} and \eqref{defubar} is established in the proof of the following lemma.

\begin{lem}
\label{KSlemosc} 
For any $n > 0$ there exist unique solutions $\hat{u}_n$ and $\bar{u}$ of  \eqref{defuhat} and \eqref{defubar} respectively. Moreover, 
\begin{equation}\label{uhtub}
\hat{u}_n \rightarrow \bar{u} \text{ in $C^1_{loc}([0,\infty))$ as $\gamma \rightarrow \infty$.}
\end{equation}
\end{lem}

\begin{proof}
Since the non-linearities are locally Lipschitz,  local existence and uniqueness of solution to \eqref{defuhat} and \eqref{defubar} 
follow from standard arguments for radial solutions. 
Also, if the solutions exist, then they are necessarily unique.
Next,  
define
$$
E_n(\rho)= \dfrac{(\hat{u}^\prime (\rho) )^2}{2}- e^{-\gamma_n}\dfrac{\hat{u}^2(\rho)}{2}+
\lambda_n e^{\hat{u}(\rho)} - e^{-\gamma_n}\gamma_n \hat{u}(\rho) \, .
$$
It is easy to check that $E_n$ is decreasing and  $E_n(0) = \lambda_n$ and since $(\lambda_n)$ converges, $|E_n(0)| \leq C$. 
Thus, since $\gamma > 0$ and $\gamma \mapsto e^{-\gamma}\gamma$ is bounded on $(0, \infty)$, Young inequality yields
\begin{equation}
((\hat{u})^2)'(\rho) \leq \hat{u}^2(\rho) + (\hat{u}')^2(\rho)  \leq E_n(\rho) + 
\hat{u}^2(\rho) 
\left( 
1 + \frac{e^{-\gamma_n}}{2} 
\right) + e^{-\gamma_n} \gamma_n \hat{u} (\rho) \leq \lambda_n + C( \hat{u}^2(\rho) + 1) \,,
\end{equation}
where $C$ is a universal constant. Then, Gronwall inequality implies
\begin{equation}
|\hat{u}(\rho)| \leq C_1 e^{C_2 \rho} \,,
\end{equation}
where $C_1, C_2$ are universal constants. Thus, $\hat{u}$ is a priori bounded, and therefore it can be uniquely extended to $[0, \infty)$. Also, since all coefficients are bounded by elliptic regularity, $\hat{u}$
 has bounded first, second, and third order derivatives locally on $[0, \infty)$,  
  uniformly in $n$. Then, by Arzel\`a-Ascoli theorem $\hat{u}_n \to \hat{u}_\infty$ in $C^2_{\textrm{loc}}[0, \infty)$. 
 Furthermore,  $e^{-\gamma_n} (\hat{u} (\rho) +\gamma_n) \to 0$ and 
 $\lambda_n e^{\hat{u} (\rho)} \to \lambda_\infty e^{\hat{u}_\infty (\rho)}$ as $n \to \infty$ locally uniformly in $\rho$. 
Thus $(\hat{u}_n)$
converges (up to sub-sequence) locally uniformly in $C^2 ([0, \infty))$ to a solution of \eqref{defubar}, and since such solution is unique, we obtain that $\bar{u} = \hat{u}_\infty$ is globally defined. Convergence 
\eqref{uhtub} follows.
\end{proof}

As in \eqref{defw}, we define $\zeta = \ln m - \ln r$ with $m = \sqrt{\frac{2(N-2)}{\lambda_\infty}}$ and we let 
$\eta (\zeta) = u(r)- 2\zeta$. Then $\eta$ satisfies (cf. \eqref{eq:fze})
\begin{equation}
\label{KSdefygamma}
\begin{cases}\eta''-(N-2)\eta' +2(N-2)( \frac{\lambda_n}{\lambda_\infty} e^{\eta}-1)-m^2 e^{-2\zeta}(\eta+2 \zeta)=0,\qquad -\infty<\zeta <\infty \,,\\ 
\lim_{\zeta \rightarrow \infty} (\eta (\zeta )+2\zeta)=\gamma_n,\\ 
\lim_{\zeta\rightarrow \infty} e^{\zeta}(\eta'(\zeta)+2)=0. \end{cases}
\end{equation}
For $\rho = e^{\frac{\gamma_n}{2}}r$
we set $\tau=\zeta-\gamma_n /2 = \ln m - \ln \rho$ and $\hat{\eta}(\tau):=\eta(\zeta)$. Observe that 
$\hat{\eta} = u(r) - 2\zeta = \hat{u}(\rho) + \gamma_n - 2\zeta = \hat{u}(\rho) - 2\tau$
is a transformed function corresponding to $\hat{u}$ solving \eqref{defuhat}. 
Standard computations show that
\begin{equation}
\label{KSdefhaty}
\begin{cases}\hat{\eta}''-(N-2)\hat{\eta}' +2(N-2)(\frac{\lambda_n}{\lambda_\infty} e^{\hat{\eta}}-1)-m^2 e^{-2\tau -\gamma_n}(\hat{\eta}+2\tau +\gamma_n) = 0,
\qquad  -\infty<\tau <\infty \,,\\ 
\lim_{\tau \rightarrow \infty} (\hat{\eta}(\tau )+ 2\tau )=0,\\ \lim_{\tau\rightarrow \infty} e^{\tau}(\hat{\eta}'(\tau)+2)=0. \end{cases}
\end{equation}
We also define   $\bar{\eta}(\tau ) = \bar{u} (\rho , \gamma) - 2\tau$, a transformed function of $\bar{u}$. Then $\bar{\eta}$ satisfies 
$$
\begin{cases}\bar{\eta}''-(N-2)\bar{\eta}' +2(N-2)(e^{\bar{\eta}}-1)=0,\ -\infty<\tau < \infty \,,\\
 \lim_{\tau \rightarrow \infty} (\bar{\eta}(\tau )+2\tau )=0,\\ \lim_{\tau\rightarrow \infty} e^{\tau}(\bar{\eta}(\tau)'-2)=0 .\end{cases}
$$
In the transformed variables,  Lemma \ref{KSlemosc} rewrites as

\begin{cor}
\label{corconvbranch}
We have
$$
\hat{\eta}_n(\cdot) \rightarrow \bar{\eta}(\cdot),\ in\ C_{loc}^1 ((-\infty ,\infty))\ \textrm{ as }\ n \rightarrow \infty .
$$
\end{cor}

\begin{proof}
For any compact 
$A \subset (-\infty, \infty)$, denote $B = \{\rho: \ln m - \ln \rho  \in A\}$ and observe that $B \subset (0, \infty)$ is bounded, 
compact, and independent of $\gamma$. Then, Lemma \ref{KSlemosc} implies that
\begin{equation}
\sup_{\tau \in A} |\hat{\eta}_n'(\tau) - \bar{\eta}'(\tau)| = 
\sup_{\tau \in A} |(\hat{\eta}_n(\tau) + 2 \tau)' - (\bar{\eta}(\tau) + 2 \tau)'| = 
\sup_{\rho \in B} |\rho(\hat{u}_n'(\rho) - \bar{u}'(\rho))| \to 0 \textrm{ as } n \to \infty \,.
\end{equation}
Analogously, we obtain  $\sup_{\tau \in A} |\hat{\eta}_n(\tau) - \bar{\eta}(\tau)| \to 0$ as $n \to \infty$ and the assertion follows. 
\end{proof}

Next, a standard calculation yields that  
$$
\bar{E}(\tau) = \frac{(\bar{\eta}'(\tau))^2}{2} + 2(N - 2) (e^{\bar{\eta}(\tau)} - \bar{\eta}(\tau) - 1) 
$$
is non-decreasing, and strictly increasing unless $\bar{\eta}'(\tau) = 0$.
Also, since $e^x - x - 1 \geq 0$, we obtain that $\bar{E}$ is bounded from below. 
A standard theory of Lyapunov functions implies that 
 $\bar{\eta}$ converges to a set of equilibria as $\tau \to -\infty$. 
Since $0$ is the only equilibrium, we have 
$(\bar{\eta}(\tau), \bar{\eta}' (\tau))\rightarrow (0,0)$ as $\tau \rightarrow -\infty$. 

Fix any $\tau_0$ and recall that $\eta_n(\zeta) = \hat{\eta}_n(\tau)$ with  $\tau = \zeta - \frac{\gamma_n}{2}$. 
Then, Corollary \ref{corconvbranch}  implies
\begin{equation}
\label{KS19sepe1}
\lim_{n \to \infty} \left(\eta_n \left(\tau_0 + \frac{\gamma_n}{2}\right),\eta_n'\left(\tau_0 + \frac{\gamma_n}{2}\right)\right) = 
\lim_{n \to \infty} \left(\hat{\eta}_n (\tau_0), \hat{\eta}_n (\tau_0)\right) = \left(\bar{\eta} (\tau_0), \bar{\eta} (\tau_0)\right) \,.
\end{equation}
Since $(\bar{\eta}(\tau), \bar{\eta}' (\tau))\rightarrow (0,0)$ as $\tau \rightarrow -\infty$, we have that the right hand side of \eqref{KS19sepe1} 
is arbitrary close to $(0, 0)$ if $\tau_0$ is large negative. 

In the following result we implicitly assume as above
that the functions depend on $n$.
Denote $z(\zeta)=\eta' (\zeta)$. Next, we show that there is $\zeta^* > 0$ independent of $n$ such that 
if $(\eta (\bar{\zeta}), z(\bar{\zeta}))$ is close to
 $(0,0)$ for some $\bar{\zeta} > \zeta^*$, then $(\eta(\zeta), \eta' (\zeta))$ is close to $(0,0)$ for any 
$ \zeta \in (\zeta^*, \bar{\zeta})$. Note that by \eqref{KS19sepe1}, $\bar{\zeta}$ is indeed large, since 
$\tau_0$ is fixed and $\gamma_n$ is large.

\begin{lem}
\label{KSlemstagamma}
For any $n > 0$ and $\varepsilon > 0$ denote 
$\Gamma_\varepsilon^n =\{(\eta,z)\in \R^2: 2(N-2)(\frac{\lambda_n}{\lambda_\infty} e^\eta -1-\eta)+\frac{1}{2}z^2 \leq \varepsilon\}$
and fix $\varepsilon_0 > 0$ such that $\Gamma_{2\varepsilon_0}^n \subset  \{\eta : |\eta|<1\}$. Note that since $\lambda_n \to 
\lambda_\infty$, $\varepsilon_0$
can be chosen idependently of $n$.
Fix $\varepsilon \in (0, \varepsilon_0)$ and 
let $\zeta^* \geq 2$ depending on $\varepsilon > 0$, but independent of $n$ be so large that 
\begin{equation}\label{pzst}
m^2\dfrac{  e^{-2\zeta}}{2} (1 + 2\zeta )^2 \leq \frac{\varepsilon}{2} \qquad \textrm{for any }  \zeta > \zeta^*\,.
\end{equation}
If there are $\bar{\zeta} > \zeta^*$ and $n > 0$
 such that $(\eta(\bar{\zeta}), \eta^\prime (\bar{\zeta})) \in \Gamma_\varepsilon^n$, then 
 $(\eta(\zeta), \eta^\prime (\zeta)) \in \Gamma_{2\varepsilon}^n$, for any $\zeta\in (\zeta^*, \bar{\zeta} )$.
\end{lem}

\begin{proof}
Fix any $n > 0$.
We set
$$
\tilde{E}_n(\eta,z,\zeta)=\dfrac{z^2 }{2}+2(N-2)\left(\frac{\lambda_n}{\lambda_\infty} e^{\eta}-1 -\eta\right)
-
\dfrac{m^2 e^{-2\zeta}}{2} (\eta +2 \zeta)^2 .
$$
Since $\eta $ satisfies \eqref{KSdefygamma}, it is easy to check that
\begin{align}
\label{convbrancprope1}
\dfrac{ d \tilde{E}_n(\eta(\zeta),\eta'(\zeta),\zeta)}{d\zeta} &=\eta^\prime \left(\eta^{\prime \prime}
+2(N-2) \left(\frac{\lambda_n}{\lambda_\infty}e^\eta -1\right) - m^2 e^{-2\zeta} (\eta+2\zeta)\right)\nonumber\\
&\qquad  -2 m^2  e^{-2\zeta} (\eta +2\zeta) +m^2  e^{-2\zeta}(  \eta +2\zeta)^2  \nonumber\\
&= (N-2)(\eta^\prime )^2 +m^2 e^{-2\zeta} (\eta+2\zeta) ( \eta + 2\zeta -2   ).
\end{align}
Fix $\varepsilon \in (0, \varepsilon_0)$ 
and let $\bar{\zeta} > \zeta^*$ be as in the statement of the lemma. Since $\varepsilon < \varepsilon_0$, then
 $\Gamma_{2\varepsilon}^n \subset \{\eta : |\eta|<1\}$
and, by continuity, $\eta(\zeta) \in \Gamma^n_{2\varepsilon}$ for any $\zeta$ close to $\bar{\zeta}$. By contradiction assume that there is $T > \zeta^*$ such that 
$$
(\eta(\zeta),z(\zeta))\in \Gamma_{2\varepsilon}^n, 
\text{ for $\zeta\in ( T,\bar{\zeta})$ and $(\eta(T),z(T))\notin \Gamma_{2\varepsilon}^n$}.
$$
Integrating \eqref{convbrancprope1} between $T$ and $\bar{\zeta}$ and recalling that $|\eta(\zeta)|\leq 1$, for $\zeta\in (T,\bar{\zeta})$ and $T > \zeta^* \geq  2$, we find
\begin{align*}
\tilde{E}_n(\eta(\bar{\zeta}),z(\bar{\zeta}),\bar{\zeta})-\tilde{E}_n(\eta(T), z(T),T) &\geq m^2  \int_{T}^{\bar{\zeta}} e^{-2\zeta} (\eta(\zeta)+2\zeta ) (\eta (\zeta )+2\zeta -2  )d\zeta \geq 0\,.
\end{align*}
Then,  recalling that $(\eta(\bar{\zeta}),z(\bar{\zeta}))\in \Gamma_\varepsilon$, we deduce from the previous line  and 
\eqref{pzst} that 
\begin{align*}
\dfrac{(z  (T))^2}{2}+2(N-2)  (e^{\eta (T)}-1 -\eta (T)) &\leq \dfrac{(z (\bar{\zeta}))^2}{2}+2(N-2)(e^{\eta(\bar{\zeta})}-1 -\eta(\bar{\zeta}) +
m^2\dfrac{  e^{-2T}}{2} (\eta(T)+2T )^2 \\
&\leq \dfrac{3}{2} \varepsilon \,,
\end{align*}
a contradiction to the definition of $T$.
\end{proof}

Now, we prove the convergence of $u_n$ to $U^\ast$ when $n \rightarrow \infty$, which 
completes the proof of Theorem \ref{KSsingintro}. 

\begin{prop}
Let $U^\ast$ be the singular solution given by Theorem \ref{KSsingintro} (cf. Proposition \ref{uniq}). 
Then, 
$$ 
u_n \rightarrow U^\ast \qquad \textrm{as } \gamma \rightarrow \infty \quad \textrm{ in }\ C^1_{loc}((0,\infty )).
$$
\end{prop}
\begin{proof}
Fix  sequences $(\gamma_n)_{n \in \N}$ and $(\lambda_n)_{n \in \N}$ with $\gamma_n \nearrow \infty$ and 
$\lambda_n \to \lambda_\infty \in (0, \infty)$ as $n \to \infty$ and let $z_n = \eta_n'$ (see \eqref{KSdefygamma}). 
Fix any small $\varepsilon \in (0, \varepsilon_0)$, where $\varepsilon_0 > 0$ and $\zeta^*$ are as in Lemma \ref{KSlemstagamma}. 
Also denote
\begin{equation}
\Gamma^*_\varepsilon = \bigcap_{n \geq 0} \Gamma^n_\varepsilon, \qquad 
\Gamma'_\varepsilon = \bigcup_{n \geq 0} \Gamma^n_\varepsilon .
\end{equation}
Since $\lambda_n \to \lambda_\infty$, $\Gamma^*_\varepsilon$ and $\Gamma'_\varepsilon$ are non-empty bounded sets that approach to 
$\{(0, 0)\}$ as $\varepsilon \to 0^+$.

By \eqref{KS19sepe1}, there exists $\tau_0 < 0$ such that, for any sufficiently large $n$, one has $\zeta_n := \tau_0 + \frac{\gamma_n}{2} > \zeta^\ast$
and $(\eta_n(\zeta_n), z_n(\zeta_n)) \in \Gamma_\varepsilon^*$. Then, Lemma \ref{KSlemstagamma} implies that
 $(\eta_n(\zeta), z_n(\zeta)) \in \Gamma_{2\varepsilon}'$, for any $\zeta \in (\zeta^*, \zeta_n]$. 
 
Since $\eta$ satisfies \eqref{KSdefygamma},  we deduce that $\eta \in C^2 ((\zeta^*, \zeta_n])$ and, after differentiating 
\eqref{KSdefygamma} with respect to $\zeta$, we obtain $\eta \in C^3 ((\zeta^*, \zeta_n])$. 
Since $\zeta^*$ is independent of $n$ and $\zeta_n \to \infty$ as $n \to \infty$, we get, by Arzel\` a-Ascoli's theorem, that
$\lambda_n \to \lambda_\infty$ and 
a standard diagonal argument shows that 
$(\eta , z)$ converges (up to sub-sequence) to $(\eta_\ast, z_\ast)$ in $(C_{loc}^1 (T,\infty ))^2$, where $(\eta_\ast (\zeta),z_\ast (\zeta))$ satisfies
$$ 
\eta_\ast^{\prime \prime}  -(N-2) \eta_\ast^\prime +2(N-2) \eta_\ast =m^2 e^{-2\zeta }(\eta_\ast+2 \zeta) - 2(N-2) (e^{\eta_\ast} -1-\eta_\ast ) ,\ \zeta\in \R .
$$
In view of the uniqueness property established in Proposition \ref{uniq}, to finish the proof, we only need to show that 
\begin{equation}
\label{propfinsubsece1}
\eta_\ast (\zeta)\rightarrow 0 \textrm{ when } \zeta\rightarrow \infty.
\end{equation} 
Suppose by contradiction that  there exists a sequence $(\zeta_k')_{k \in \N}$ such that 
$\zeta_k' \rightarrow \infty$, as $k\rightarrow \infty$ and a constant $\delta>0$ such that 
\begin{equation}
\label{finsebcontr}
(\eta_\ast (\zeta_k'), z_\ast (\zeta_k')) \notin \Gamma_\delta', \textrm{ for any }  k\geq 1 \,.
\end{equation}
By decreasing $\varepsilon$ if necessary, we can suppose that $\varepsilon \leq \delta /4$. 
Choose $k$ sufficiently large such that $\zeta_k' > \zeta^*$, $\tau_0 < 0$ and sufficiently large $n$ such that  $\zeta_n > \zeta_k'$ and 
$(\eta(\zeta_n, \gamma_n ), z(\zeta_n, \gamma_n)) \in \Gamma_\varepsilon^*$ (cf. \eqref{KS19sepe1}). 
Then, by Lemma \ref{KSlemstagamma} one has 
$(\eta(\zeta, \gamma_n ), z(\zeta, \gamma_n)) \in \Gamma_{2 \varepsilon}' \subset \Gamma_{\delta}'$ 
for any $\zeta \in (\zeta^*, \zeta_n)$, a contradiction.

Overall, we proved that $\eta(\cdot, \gamma) \to \eta_\infty$ in $C^1_{\textrm{loc}}(\R)$, where $\eta(\cdot, \gamma)$ solves \eqref{KSdefygamma} and $\eta_\infty$ 
satisfies \eqref{eq:fze} with \eqref{KSuniqeq}. Finally, fix any open set $A$ with $\bar{A} \subset (0, \infty)$ and let $B := \{\zeta \in \R : \ln m - \ln r \in A\}$. Since $B$
is open and bounded, one has, for some constant $C_A$ depending on $A$, 
\begin{align*}
 \|u(\cdot, \gamma) - U^*(\cdot)\|_{C^1 (A)} &= \|(u(\cdot, \gamma) -2( \ln m - \ln \cdot)) - (U^*(\cdot ) - 2(\ln m - \ln \cdot))\|_{C^1 (A)} \\
&\leq  C_A \|\eta(\cdot, \gamma) - \eta_{\infty}(\cdot)\|_{C^1 (B)} \to 0 \qquad \textrm{as } \gamma \to 0 \,,
\end{align*}
as desired. 
\end{proof}

\section{Oscillation of the branch and Morse index: proof of Theorem \ref{KSoscbranch} and Proposition \ref{KSmorse}.}
\label{sec4}
To prove Theorem \ref{KSoscbranch},  we first recall a result of Joseph and Lundgren \cite{josephlundgren}. 
Let
\begin{equation}
\label{defubarast}
\bar{u}^\ast (r)= -2\log r +k, \qquad k=\log \dfrac{2(N-2)}{\lambda} 
\end{equation}
be the singular solution of \eqref{defubar}, that is, it satisfies the equation in \eqref{defubar} and blows-up at the origin. 

\begin{prop}
\label{gelfwellknown}
For any $\alpha \geq 0$, let $\bar{u}(\cdot,\alpha)$ (resp. $\bar{u}^\ast$) 
be defined in \eqref{defubar} (resp. \eqref{defubarast}).
Then,
$$
Z_{[0,\infty )}[\bar{u}(\cdot, \alpha) -\bar{u}^\ast (\cdot) ]=
\begin{cases} 
\infty &\textrm{if }  3\leq N\leq 9 \\ 
0 & \textrm{if }  N\geq 10, 
\end{cases}
$$
where $Z_I (u)=\sharp \{r\in I| u(r)=0\}$ and $\sharp A$ is the cardinality of the set $A$.
\end{prop}

For any given $\gamma > 0$, let $(r_{\lambda,\gamma}^i)$ be an increasing (finite or infinite) sequence of positive real numbers such that $u'(r_{\lambda,\gamma}^i , \gamma) =0$, where $u(\cdot, \gamma) = u_\lambda (\cdot, \gamma)$ is the unique solution of 
\eqref{KSeqintfin}. 
We show that if $3\leq N \leq 9$, then $r_{\lambda,\gamma }^i $ oscillates around $R_\lambda^i$ (recall that $(U^*)'(R_\lambda^i) = 0$) infinitely many times as $\gamma \rightarrow \infty$.

The following main result of this section is partly motivated by \cite[Lemma 5]{miya2}, where a problem with Dirichlet boundary
conditions is considered (see also \cite{miya1} for a related problem with Neumann boundary condition). However, in the works above, 
it is assumed that the parameter can be removed from the equation by rescaling of the domain. 
Our situation is different and we have to work directly with parameter dependent equation. We also have to appropriately modify
the zero number argument to treat Neumann boundary conditions. 

\begin{lem}
Assume $3 \leq N \leq 9$ and fix $R > 0$. 
If $\lambda^i$ be the positive real number given in Theorem \ref{KSmainthm}, then there exist a sequence of initial data $(\gamma_n)_n$ with $\gamma_n \to \infty$ and a sequence positive integer $(j_n)_n$
 such that $r_{\lambda^i,\gamma_n}^{j_n} = R$. In other words, $u(\cdot, \gamma_n)$ satisfies Neumann boundary data on $\partial B_R$. 
\end{lem}

\begin{proof}
First, for any $\lambda > 0$ we show that, for any $A > 0$ and $I = (0, A)$, one has
\begin{equation}\label{znum}
Z_{I}[u (\cdot, \gamma ) - U^\ast (\cdot)]\rightarrow \infty \textrm{ as } \gamma \rightarrow \infty 
\,.
\end{equation}
Recall that, by Lemma \ref{KSlemosc}, we have
\begin{equation}
\label{KSlemosce1}
\hat{u}(\rho , \gamma )\rightarrow \bar{u}(\rho , 0) \text{ in $C^1_{loc}([0,\infty ))$ when $\gamma\rightarrow \infty$} \,,
\end{equation}
where $\rho = e^{\frac{\gamma}{2}}r$,  $\bar{u} (\rho, 0) = \bar{u}(r, \gamma) - \gamma$ satisfies \eqref{defubar} and $\hat{u}(\rho, \gamma) = u(r, \gamma) - \gamma$  satisfies \eqref{defuhat}.  
Set $\hat{U}^\ast (\rho ,\gamma ) = U^\ast (r) -\gamma$ and $k = \ln \frac{2(N-2)}{\lambda}$. The condition on 
$U^*$ at the origin yields that, for any $r_0 > 0$, there is $C = C(r_0)$ such that 
\begin{equation}
 |U^*(r) + 2\ln r - k | \leq C r \qquad \textrm{for any } r \leq r_0\,,
\end{equation}
and therefore
\begin{equation}
|\hat{U}^\ast (\rho , \gamma) + 2 \ln \rho -  k| \leq C \rho e^{-\frac{\gamma}{2}} \qquad \textrm{for any } \rho \leq r_0 e^{\gamma/2} \,.
\end{equation}
Consequently, 
\begin{equation}
\label{KSlemosce2}
\hat{U}^\ast (\cdot ,\gamma )  \rightarrow \bar{u}^\ast  \textrm{ in } C_{loc} ((0,\infty )) \textrm{ as } \gamma \rightarrow \infty, 
\end{equation}
where $\bar{u}^\ast$ is defined in \eqref{defubarast}. Fix any $M > 0$. Then, by Proposition \ref{gelfwellknown}, there exists a bounded interval $I_M \subset (0, \infty)$  such that 
\begin{equation}
Z_{I_M} [\bar{u}(\cdot, 1) - \bar{u}^*(\cdot)] \geq M.  
\end{equation}
By scale invariance of the equation, one has $\bar{u}^*(r) = \bar{u}^*(e^{\alpha/2}r) + \alpha$ and $\bar{u}(r, 1 + \alpha) 
= \bar{u}(e^{\alpha/2}r, 1) + \alpha$, and therefore, for any $\gamma \geq 1$,
\begin{equation}
Z_{I_M} [\bar{u}(\cdot, \gamma) - \bar{u}^*(\cdot)] = 
Z_{e^{\frac{\gamma - 1}{2}}I_M} [\bar{u}(\cdot, 1) - \bar{u}^*(\cdot)] \geq M.  
\end{equation}
Then, thanks to \eqref{KSlemosce1} and \eqref{KSlemosce2}, we have
$$
Z_{I_M} [\hat{u}( \cdot, \gamma ) - \hat{U}^\ast (\cdot, \gamma ) ]
\geq Z_{I_M} [\bar{u}(\cdot, \gamma) - \bar{u}^*(\cdot)] \geq M  \,.
$$
Finally, for given $I$ and sufficiently large $\gamma$ one has $I_M \subset e^{\frac{\gamma}{2}} I$, and consequently
$$
Z_{I} [u( \cdot, \gamma ) - U^\ast ( \cdot, \gamma ) ] 
= Z_{e^{\frac{\gamma}{2}} I} [\hat{u}( \cdot, \gamma ) - \hat{U}^\ast (\cdot, \gamma ) ]
\geq M \,.
$$
Since $M$ was arbitrary, the claim \eqref{znum} follows. 

For  $\lambda := \lambda^i$, let $U^\ast$ be the solution of \eqref{KSwhole} and notice that $(U^*)'(R) = 0$. Also, for the same
$\lambda$, let $u(\cdot, \gamma)$ be the solution of \eqref{KSeqintfin}. Observe that $u(\cdot, \gamma)$ does not necessarily 
satisfy Neumann boundary condition at $R$. Since $w_\gamma := u(\cdot, \gamma) - U^*$ satisfies a linear differential equation, 
it follows from the uniqueness of initial value problem for ODEs that every zero of $w_\gamma$ is simple. 

Observe that, for every $\gamma > 0$, $Z_{[0, 1]} (w_\gamma) < \infty$ since otherwise by continuity, the accumulation point would be 
a degenerate zero. Also, since $w_\lambda$ has only finitely many simple zeros, continuous dependence on parameters yields that 
zeros of $w_\gamma$ depend continuously on $\gamma$. For each $\gamma > 0$, let $m_\gamma := Z_{[0, 1]} (w_\gamma)$ and let 
$(z_j^\gamma)_{j = 1}^{m_\gamma} \subset[0, R]$ be the increasing sequence of zeros of $w_\gamma$. 
Since $w_\gamma (0) = -\infty$, we have that $z_1^\gamma > 0$
 for each $\gamma > 0$, and moreover $w'_\gamma(z_1^\gamma) > 0$.
By induction it is easy to prove that $w'_\gamma(z_i^\gamma) > 0$ if $i$ is odd and $w'_\gamma(z_i^\gamma) < 0$ if $i$ is even. 

Since the zeros of $w_\gamma$ are non-degenerate a new zero of $w_\gamma$ cannot be created in the interior of $[0, R]$. Furthermore,  
$w_\gamma(0) = - \infty$, and therefore a new zero cannot enter $[0, R]$ through $0$.  Hence, \eqref{znum} yields that there exists a sequence $(\gamma_k)$ with $\gamma_k \to \infty$
as $k \to \infty$ such that $w_{\gamma_k}(R) = 0$. Since $w'_{\gamma_k} (R) > 0$ if $k$ is odd and $w'_{\gamma_k} (R) < 0$ if $k$ is even, 
by the continuous dependence on parameters, we obtain that there exists $\gamma^*_k \in (\gamma_k, \gamma_{k + 1})$ such that 
$w'_{\gamma_k^*} (R) = 0$. Since $(U^*)'(R) = 0$, we infer that $u'(R, \gamma^*_k) = 0$ and the lemma follows.
\end{proof}

Next, we prove that the Morse index of the singular solution $U^\ast_\lambda$  is finite when $N> 10$ and infinite when $3\leq N\leq 9$.

\begin{proof}[Proof of Proposition \ref{KSmorse}.]
Assume $3 \leq N \leq 9$. In order to prove that $U^*_{\lambda^i}$ has infinite Morse index, by variational characterization of eigenvalues, it suffices to prove that 
there are infinitely many linearly independent functions $f : (0, 1) \to \R$ such that 
\begin{equation}\label{dffj}
\mathcal{J}(f) = \int_0^1 \left(|f'|^2  + (1- \lambda^i e^{U^*_{\lambda^i}}) f^2\right) r^{N - 1} dr < 0 \,.
\end{equation}
By the boundary conditions \eqref{KSwhole},  we see that, for any $\varepsilon>0$, there exists $r_0$ such that, for all $r\in (0,r_0)$,
$$
\lambda^i e^{U_{\lambda^i}^\ast (r)}-1 \geq \dfrac{2(N-2)}{r^2}(1-\varepsilon).
$$
Then, it follows that if $3\leq N\leq 9$, we have, for some small $\varepsilon_0>0$,
\begin{equation}\label{upbds}
\lambda^i e^{U_{\lambda^i}^\ast (r)}-1 \geq  \left(\dfrac{(N-2)^2}{4}+\varepsilon_0^2\right)\dfrac{1}{r^2}.
\end{equation}
Next, we define $f_j (r)=f(r) \tilde{\chi}_j (r)$, where 
$$
\tilde{\chi}_j (r)=
\begin{cases}1,& \text{ if $r\in [r_{j+1},r_j]$,}\\ 
0, & \text{ elsewhere } ,
\end{cases}
\qquad 
r_j=e^{-2\pi j/\varepsilon_0}
$$
and $f(r)=r^{-(N-2)/2} \sin (\varepsilon_0 \log r /2)$. Notice that $f_j$ and $f_k$ have disjoint supports for 
$j \neq k$, and therefore they are linearly independent. Moreover, $f_j$  is a solution of 
$$
- f_j'' - \frac{N - 1}{r} f_j' -  \left(\dfrac{(N-2)^2}{4}+\dfrac{\varepsilon_0^2}{4}\right) \dfrac{1}{r^2} f_j=0,\quad r\in (r_{j+1},r_j).
$$
Since $f_j(r_j) = f_j(r_{j + 1}) = 0$ we have that $f_j \in W^{1, 2}((0, \infty))$ and by \eqref{upbds}
\begin{align}
\mathcal{J}(f_j) = 
\int_{r_{j+1}}^{r_j} \left(|f_j'|^2 -  \left(\dfrac{(N-2)^2}{4}+\varepsilon_0^2\right)\dfrac{1}{r^2} f_j^2\right)r^{N-1} dr = -\dfrac{3}{4}\varepsilon_0^2 \int_{r_{j+1}}^{r_j} \dfrac{1}{r^2} f_j^2 dx <0 \, .
\end{align}
Thus the Morse index of $U^*_{\lambda_i}$ is infinite.

Next, let us consider the case $N> 10$. We show that there are at most finitely many linearly independent functions satisfying \eqref{dffj}.
Recall that $(U^\ast_{\lambda_i} )^\prime (R)=0$. Again, by using asymptotics of $U^*_{\lambda^i}$ at the origin, we have 
that, for $\varepsilon >0 $,  there exists $r_0\in (0,1)$ such that, for any $r\in (0,r_0)$,
\begin{equation}\label{uubb}
\lambda^i e^{U_{\lambda^i}^\ast (r)}-1 \leq \dfrac{2(N-2)}{r^2}(1+\varepsilon)\leq  \dfrac{(N-2)^2}{4 r^2} \,,
\end{equation}
where the last inequality holds for $N > 10$. 
Next, choose  $\chi_0 \in C^1 (\R^N)$ such that 
$$
\chi_0 (r)=
\begin{cases}1, & \text{ if $r\in (0, r_0 /2)$},\\ 0,& \text{ if $r> r_0 $,}
\end{cases}
$$
and set $\chi_1 =1-\chi_0$. For $\phi \in H^1_{rad}(B_1 (0))$ with $\phi '(R)=0$, the Hardy inequality 
\cite{Hardy} and \eqref{uubb} imply
\begin{align*}
\mathcal{J}(\phi) &= \int_{0}^1 (|\phi' |^2 -(\chi_0 +\chi_1) (\lambda^i e^{U_{\lambda^i}^\ast (r)}-1) \phi^2 )r^{N - 1} dr\\
&\geq  \int_{0}^1 \left(|\phi' |^2  - \chi_0 \dfrac{(N-2)^2}{4 r^2}  \phi^2 \right)r^{N-1} dr 
+ \int_{0}^1 (|\phi'|^2 - \chi_1 (\lambda^i e^{U_{\lambda^i}^\ast (r)}-1) \phi^2 )r^{N-1} dr \\
&\geq \int_{0}^1 (|\phi' |^2 -\chi_1 (\lambda^i e^{U_{\lambda^i}^\ast (r)}-1) \phi^2 )r^{N-1} dr
\end{align*}
Since $|(\lambda^i e^{U_{\lambda^i }^\ast (r)}-1)|\leq C_{\lambda^i}$, for $r\in (r_0 /2 , 1)$,  the operator $-\Delta  -\chi_1 (\lambda^i e^{U^\ast_{\lambda^i} (r)}-1) $ 
on $B_1(0)$
with Neumann boundary condition has finitely may negative eigenvalues, and therefore 
\begin{equation}
\int_{0}^1 (|\phi' |^2 -\chi_1 (\lambda^i e^{U_{\lambda^i}^\ast (r)}-1) \phi^2 )r^{N-1} dr < 0
\end{equation}
has only finitely many linearly independent solutions. Thus, the Morse index of $U_{\lambda^i}^\ast $ is finite as desired.
\end{proof}

\section{Proof of Theorem \ref{KSmainthm} }
\label{sec5}

In this section, we prove Theorem \ref{KSmainthm}. Let $(R_\lambda^i )_{i=1}^ \infty$, 
 be an increasing, unbounded sequence of positive real numbers depending on $N$ and 
$\lambda$ such that $(U^\ast_\lambda)^\prime (R_\lambda^i)=0$ 
(see Lemma \ref{corbehazero}), where $U^\ast_\lambda$ is the solution to \eqref{KSwhole}. 
To prove Theorem \ref{KSmainthm}, we need two ingredients. First we show that, 
 for any $i \in \N$,
\begin{equation*}
\label{KSMlabdalimit}
R_\lambda^i \to 0, \text{ as $\lambda\to 0^+$}
\end{equation*}
and obtain necessary bounds on solutions.  
Then, we show that, for any $i\in \N$, 
 the map $\lambda\rightarrow R_\lambda^i$ is continuous.

\begin{prop}
\label{KSpinfty}
For each $\lambda > 0$, let $U^\ast_\lambda$ be the unique solution to \eqref{KSwhole} and denote by $(R_\lambda^i)_{i=1,\ldots ,\infty}$, the increasing sequence of all positive real numbers such that $(U_\lambda^\ast )^\prime (R_\lambda^i )=0$. Then, for any fixed $i\in \N$, we have
$$R_\lambda^i \to 0,\ as\ \lambda\to 0^+.$$
\end{prop}

\begin{proof}
The proof is divided into several steps. We begin by giving some notations. Many constants and functions in the proof depend on $\lambda$ and
\begin{equation}
m = \sqrt{\frac{2(N-2)}{\lambda}} \, .
\end{equation}
However for the 
clarity of the notation, this dependence is not explicitly indicated, but the needed asymptotic is explained. 
If a constant depends only on the dimension $N$, we usually denote it by $C_N, c_N$, etc. Note that such constant can change from line 
to line. First, we define 
\begin{align}
f(\zeta) &= \frac{m^2 }{2(N-1)} e^{-2 \zeta } \left(\zeta + \dfrac{N-2}{4(N-1)} \right).
\end{align}
and let $\zeta \mapsto \eta$ be the unique solution of \eqref{eq:fze} (see Proposition \ref{uniq}). 
Setting  $\tilde{\eta}(\zeta)  = \eta(\zeta) - f(\zeta)$, 
we see that $\eta$ satisfies
\begin{equation}\label{eq:neq}
\tilde{\eta}^{\prime \prime} - (N-2) \tilde{\eta}^\prime +2(N-2) \tilde{\eta} = m^2 e^{-2\zeta} \eta(\zeta ) +\phi (\eta(\zeta))=: \tilde{g}(\zeta), 
\end{equation}
where $\phi$ is as in \eqref{dfph}.

Define $\alpha, \beta$, and $G_N$ as respectively in \eqref{dfab} and \eqref{dfgf} 
and recall that  $G_N \in L^1(\mathbb{R}) \cap L^\infty(\mathbb{R})$ for any $N \geq 3$. Hence, 
\begin{equation} \label{eqdf}
\tilde{\eta}(\zeta) =  \int_\zeta^\infty G_N(\sigma - \zeta) \tilde{g}(\sigma) \, d\sigma =  G_N\ast \tilde{g}(\zeta) .
\end{equation}
If $v := U^\ast /\overline{u}_\lambda$, then, $v$ satisfies
\begin{equation}\label{eq:vag}
-v'' - \frac{N-1}{r} v' + v = e^{\overline{ u}_\lambda (v - 1)} 
\end{equation}
and $w(r)=r^{\frac{N-1}{2}}(v (r) - 1)$ satisfies (see the proof of Lemma \ref{corbehazero})
\begin{equation}\label{eq:wag}
w^{\prime \prime} + \left(\dfrac{ e^{\overline{u}_\lambda (v -1)} - v}{v -1 }-\dfrac{ (N-1)(N-3)}{4r^2}  \right)w = 0.
\end{equation}
 For any $\lambda \in (0, 1/e)$, we recall that $\overline{u}_\lambda > 1$ is the solution of the equation $u = \lambda e^u$.
Let $r_\lambda$ be the smallest $r$ such that $U^\ast (r) = \overline{u}_\lambda$, or equivalently the smallest point such that $v(r) = 1$ or $w(r) = 0$.

\noindent
\textbf{Step 1.} For any $\delta > 0$, we have
\begin{equation}
  - (1 - \delta) \ln \lambda \leq  \overline{u}_\lambda ,
\end{equation}  
for any sufficiently small $\lambda$ depending on $\delta$. 

\begin{proof}[Proof of Step 1.]
 By taking the logarithm of the equality $\overline{u}_\lambda = \lambda e^{\overline{u}_\lambda}$, we obtain
\begin{equation}
\ln \overline{u}_\lambda -\overline{u}_\lambda  - \ln \lambda = 0\,.
\end{equation}
For $ v_1 = - (1-\delta)\ln \lambda$, we have
\begin{equation}
\ln v_1 - v_1 - \ln \lambda =   \ln ((1-\delta ) \ln \lambda^{-1}) -\delta \ln \lambda > 0 ,
\end{equation}
for any sufficiently large $\lambda$ depending only on $\delta$. On the other hand, for any fixed $\lambda$ and sufficiently large $v$, one has 
\begin{equation}
\ln v - v - \ln \lambda < 0 \,.
\end{equation}
In particular there is a solution of $u = \lambda e^{u}$ which is bigger than $v_1 = - (1-\delta)\ln \lambda$. Finally, since $\overline{u}_\lambda$ is the biggest solution, 
$\bar{u}_\lambda \geq v_1$, and the claim follows. 
\end{proof}

\begin{rmq}
For any $\delta > 0$, one can prove more the precise bound
\begin{equation}
-\ln \lambda + \ln(-\ln \lambda) < \overline{u}_\lambda \leq -(1 + \delta) \ln \lambda \, ,
\end{equation}
for any sufficiently small $\lambda$ depending on $\delta$.
\end{rmq}

\noindent
\textbf{Step 2.} 
Recall that $r_\lambda$ is the smallest $r$ such that $U^*(r) = \bar{u}_\lambda$. Then, 
there exists $K_N > 0$ such that $r_\lambda^2 < \frac{K_N}{\overline{u}_{\lambda}}$, for any small $\lambda > 0$.

\begin{proof}[Proof of Step 2.]
Set $K_N = \max\{16(N - 1)(N-3), 2(16 \pi)^2\}$. For a contradiction, assume that there exists a sequence  $\lambda_n \to 0$ as $n \to \infty$  such that $r_n^2  \geq K_N /\overline{u}_n $, where 
$r_n := r_{\lambda_n}$ and $\overline{u}_n := \overline{u}_{\lambda_n}$. Then $w_n := w_{\lambda_n} > 0$  and $v_n := v_{\lambda_n} > 1$
(solutions of \eqref{eq:wag} and \eqref{eqdf} with $\lambda = \lambda_n$)  on 
$I_n :=  [A_n, 2A_n]$ with $A_n = \sqrt{K_N/(16\overline{u}_n)}$ for any $n$. Since for any $x \geq 0$ one has $e^x  \geq  x + 1$ and $v_n > 1$, we have 
for any $r \in I_n$ 
\begin{equation}
\dfrac{ e^{\overline{u}_n (v -1)} - v}{v -1 }-\dfrac{ (N-1)(N-3)}{4r^2}  \geq \dfrac{ \overline{u}_n (v -1) + 1 - v}{v -1 } - \dfrac{  \overline{u}_n (N-1)(N-3)}{4K_N} 
\geq  \frac{3}{4} \overline{u}_n - 1 \,,
\end{equation} 
where the last inequality holds by the definition of $K_N$.
Furthermore, by Step 1, $\bar{u}_\lambda \to \infty$, and therefore it is possible to choose $n$ large enough such that 
\begin{equation}
\frac{3}{4}\overline{u}_n - 1 >   \frac{1}{2}\overline{u}_n \geq \frac{(16 \pi)^2 \overline{u}_n}{K_N}  \geq \frac{ (4\pi)^2}{A_n^2}.
\end{equation}
Then, $w_n$ satisfies
\begin{equation}
w_n'' + q_n(r)w_n = 0,  \text{ on } [A_n,  2A_n],
\end{equation}
with $q_n > (4\pi)^2/A_n^2$ for any sufficiently large $n$. However, the equation
\begin{equation}
m'' + \frac{(4\pi)^2}{A_n^2}m = 0
\end{equation}
has a solution $m(r) = \sin(4\pi r/A_n)$ which has zeros at $A_n + \frac{k}{4}A_n \in [A_n, 2A_n]$ for any $k \in \{0, 1, \cdots, 4\}$. By the Sturm-Piccone comparison theorem, $w_n$ has also a zero 
on $I_n$, contradicting the fact that $w_n > 0$ on $I_n$. 
\end{proof}

\noindent

Let $r_\lambda$ be as in Step 3. and let $\zeta_\lambda $ be defined by  (see \eqref{nrv})
 \begin{equation}
\label{relrzeta}
r_\lambda =\sqrt{\dfrac{2(N-2)}{\lambda}} e^{-\zeta_\lambda}.
\end{equation}

\noindent
\textbf{Step 3.}
There exists a constant $C_N$ such that 
for any sufficiently small $\lambda > 0$,  one has $f(\zeta_\lambda) \leq C_N$.

\begin{proof}[Proof of Step 3.]
Step 1 and Step 2 with $\delta = \frac{1}{2}$ imply for any small $\lambda > 0$
\begin{equation}
r_\lambda^2 \leq \frac{K_N}{\overline{u}_{\lambda}} \leq \frac{2K_N}{- \ln \lambda}\,,
\end{equation}
which is equivalent to
\begin{equation}
e^{-2\zeta_\lambda} \leq \dfrac{K_N}{N-2} \frac{\lambda}{-\ln \lambda} .
\end{equation}
The previous inequality can be rewritten as
\begin{equation}
 \zeta_\lambda \geq - \frac{1}{2} \left(\ln \left(\dfrac{K_N}{N-2} \lambda\right) - \ln(-\ln \lambda) \right) \,.
\end{equation}
In particular, we see that $\zeta_\lambda \to \infty$ as $\lambda \to 0$.
Since the function $x \mapsto xe^{-x}$ is decreasing on $(0, \infty)$, for sufficiently small $\lambda > 0$,
\begin{align}
f(\zeta_\lambda) &= \dfrac{N-2}{N-1}  \frac{ e^{-2\zeta_\lambda}}{\lambda} \left(\zeta_\lambda+ \dfrac{N-2}{4(N-1)} \right) 
\leq - \dfrac{K_N}{2 N-2} \frac{ \left(\ln(\frac{K_N}{ N-2} \lambda) - \ln(-\ln \lambda)\right) \frac{\lambda}{-\ln \lambda}}{\lambda} 
\leq C_N.
\end{align}
 This proves Step 3.
\end{proof}

\begin{rmq}\label{rmk:zti}
For clarity let us indicate explicitly the dependence of $f$ on $\lambda$ (or equivalently on $m$). 
Fix any $M > 0$ and for each $\lambda > 0$ choose $\bar{\zeta}_\lambda \geq 0$ such that $f_\lambda (\bar{\zeta}_\lambda) \leq M$. 
Since $\inf_K f \to \infty$ as $m \to \infty$ on any compact set $K \subset (0, \infty)$, one has
\begin{equation}
 \bar{\zeta}_\lambda \to \infty , \qquad \textrm{as } \lambda \to 0 \quad \textrm{ or equivalently if } \quad  m \to \infty \,.
\end{equation} 
We frequently use this observation below, often without further reference. 
\end{rmq}

Next, we derive estimates on $\tilde{\eta}$ solving \eqref{eq:neq} . We consider two cases:  $f(\zeta_\lambda )\leq 1.1$  and  $f(\zeta_\lambda )\geq 1.1$, where $\zeta_\lambda$ is given by \eqref{relrzeta}.

\noindent \textbf{Step 4.}
There exists a constant $\tilde{C}_N$ such that if, for sufficiently small $\lambda > 0$, one has
 $1.1 \leq f(\zeta_\lambda)$,  then $|\tilde{\eta} (\zeta_\lambda)| \leq \tilde{C}_N$.

\begin{proof}
First, the assumption and  Step $3$ yield $1 \leq f(\zeta_\lambda) \leq C_N$, and therefore by Remark \ref{rmk:zti},  $\zeta_\lambda \to \infty$ as $\lambda \to \infty$. 
Hence, there exist two constants $c_N<1<C_N$ such that, for any sufficiently small $\lambda$,
\begin{equation}
c_N \lambda \leq e^{-2\zeta_\lambda} 2\zeta_\lambda \leq C_N \lambda .
\end{equation}
Using that $\lambda = \overline{u}_\lambda e^{-\overline{u}_\lambda}$ and $\overline{u}_\lambda \to \infty$ as $\lambda \to 0$ (see Step 1), 
we obtain that for small  $\lambda$
\begin{equation}
C_N \lambda = \overline{u}_\lambda e^{-\overline{u}_\lambda + \ln C_N} \leq  (\overline{u}_\lambda - 2\ln C_N) e^{-\overline{u}_\lambda + 2\ln C_N} \, ,
\end{equation}
and 
\begin{equation}
c_N \lambda = \overline{u}_\lambda e^{-\overline{u}_\lambda + \ln c_N} \geq  (\overline{u}_\lambda - 2\ln c_N) e^{-\overline{u}_\lambda + 2\ln c_N} \, .
\end{equation}
Consequently
\begin{equation}
(\overline{u}_\lambda - 2\ln c_N) e^{-\overline{u}_\lambda + 2\ln c_N} \leq 
c_N \lambda \leq e^{-2\zeta_\lambda} 2\zeta_\lambda \leq C_N \lambda
\leq  (\overline{u}_\lambda - 2\ln C_N) e^{-\overline{u}_\lambda + 2\ln C_N} \,.
\end{equation}
Since the function $x \mapsto xe^{-x}$ is decreasing on $(0, \infty)$, we have
\begin{equation}
\overline{u}_\lambda - 2\ln c_N \geq 2\zeta_\lambda \geq \overline{u}_\lambda - 2\ln C_N \,.
\end{equation}
Recalling that, by definition,
\begin{equation}
\overline{u}_\lambda = u(r_\lambda) = f(\zeta_\lambda) + \tilde{\eta}(\zeta_\lambda) + 2 \zeta_\lambda,
\end{equation}
we deduce that
\begin{equation}
-2 \ln C_N \leq  f(\zeta_\lambda) + \tilde{\eta}(\zeta_\lambda) \leq -2 \ln c_N   .
\end{equation}
Since $1 \leq f(\zeta_\lambda) \leq C_N$, we obtain the desired result.
\end{proof}

 Before proceeding let us introduce some additional notation. Define  
\begin{equation}
\Gamma = 1.1, 
\end{equation}
and denote $\zeta_1^*$ the largest solution of $f(\zeta) = \Gamma$, where of course $\zeta_1^*$ depends on $\lambda$ and  by Remark \ref{rmk:zti}, $\zeta_1^* \to \infty$ as $\lambda \to \infty$. We remark that instead of 1.1, we can
take any number bigger than 1, sufficiently close to 1.  

 Fix any $\varepsilon_0 > 0$ and 
set 
\begin{equation}\label{defzeta2ast}
\zeta_2^* := \inf\{ \zeta\geq \zeta_1^* : |\eta(z)| \leq (1+ \varepsilon_0) P_N f(z) \textrm{ for any } z \geq \zeta\},
\end{equation} 
with $\inf \emptyset = \infty$. Denote
\begin{equation}
P_N := 
\frac{\left|\phi\left(\Gamma \right) \right|}{ \Gamma} \tilde{P}_N :=
 \frac{\left|\phi\left(\Gamma \right) \right|}{ \Gamma} \times 
\begin{cases} 
 (1 + e^{-\frac{(\alpha + 8) \pi}{2\beta}} ) \frac{4}{(\alpha + 8)^2 + 4\beta^2}& \textrm{if } \ 3\leq N\leq 9,\\ 
  \frac{4}{(\alpha + 8)^2} & \textrm{if } N=10,\\ 
 \frac{4}{(\alpha + 8)^2 - 4\beta^2} & \textrm{if } N>10\, ,
 \end{cases}
\end{equation}
where $\phi$ is defined in \eqref{dfph}.
Clearly, $P_N$ and $\tilde{P}_N$ are constants depending only on $N$ and $\zeta_2^*$ depends on the solution $\eta$. Since $\zeta_2^* \geq \zeta_1^*$, one has 
\begin{equation}
\zeta_2^* \to \infty \qquad \textrm{as } m \to \infty \,. 
\end{equation}
Moreover,
\begin{equation}
P_N = \frac{e^\Gamma - \Gamma - 1}{(N - 2)\Gamma} \tilde{P}_N <  \frac{e^\Gamma - \Gamma - 1}{3 \Gamma}  < \frac{1}{3}\,,
\end{equation}
where in the first inequality, after standard manipulations, we used that $N \mapsto \tilde{P}_N$ is increasing and $\tilde{P}_N \to 1/3$ as $N \to \infty$. 
In particular, for any  $\varepsilon_0 > 0$
one has 
$|\eta (\sigma)|  \leq P_{N, \varepsilon_0}  f(\sigma) $ for each $\sigma \in (\zeta_2^*, \infty)$, where $P_{N, \varepsilon_0} := P_N(1 + \varepsilon_0)$.

\medskip

Next, in the following three steps we obtain estimates on $\tilde{\zeta}$ on the interval $[\zeta_1^*, \infty)$ and in particular we 
prove that Step 4 remains valid if $f(\zeta_\lambda )\leq 1.1$.

\noindent {\bf Step 5:} For any $m > 0$ and $\varepsilon_0 >0$,  one has $\zeta_2^* < \infty$.

\begin{proof}[Proof of Step 5]
We proceed as in the proof of Lemma \ref{behaorigin}.
Using the representation formula \eqref{eqdf} and Young's inequality for convolutions, we obtain 
\begin{equation}\label{imca}
\int_\zeta^\infty |\tilde{\eta} (\sigma )|d\sigma \leq C_N \int^\infty_{\zeta}   |\tilde{g}(\sigma)|  d\sigma \,,
\end{equation}
where $C_N = \|G_N\|_{L^1}$. 
Since $\eta (\zeta)\to 0$ as $\zeta \to \infty$,  for 
any $\varepsilon > 0$  there is $\zeta_0 > 0$ depending on $N$, $\eta$, and $\varepsilon$ such that for any $\zeta \geq \zeta_0$ one has 
\begin{align}
\label{KSdefzeta01}
 |\phi(\eta)| 
&= 2(N-2) \left| e^{\eta} - 1 - \eta  \right| \leq \frac{\varepsilon}{2} |\eta|\,, \qquad 
\textrm{where } \eta = \eta(\zeta) \,.
\end{align}
By the definition of $\tilde{g}$ and \eqref{KSdefzeta01}, one has, for $\sigma \geq \zeta_0$, 
\begin{equation*}
|\tilde{g}(\sigma)| \leq  m^2 e^{-2\sigma} (|f(\sigma)|  +  |\tilde{\eta}(\sigma)|) +\frac{ \varepsilon}{2} (|\tilde{\eta}(\sigma)| + |f (\sigma)|) \,. 
\end{equation*}
Fix  $\overline{\zeta}_0 > 0$  such that $ m^2  e^{-2 \overline{\zeta}_0} \leq  \frac{\varepsilon}{2} $, 
and set $\zeta^*_0=\max\{\zeta_0 , \overline{\zeta}_0 \}$. Then we have, for $\sigma \geq \zeta_0^*$,
\begin{align}
\label{KSlem1e2o}
|\tilde{g}(\sigma)| &\leq  \varepsilon ( |f(\sigma)|+  |\tilde{\eta}(\sigma)|) 
\,.
\end{align}
Substituting \eqref{KSlem1e2o} into \eqref{imca} and requiring that $\varepsilon \in(0, 1/(2C_N))$ 
we obtain, for $\zeta \geq  \zeta_0^* > 0$, 
\begin{equation}
 \int_{\zeta}^\infty |\tilde{\eta}(\sigma)|\, d\sigma \leq  
   \frac{1}{2} \int_{\zeta}^\infty |f(\sigma)| \, d\sigma +  \frac{1}{2} \int_{\zeta}^\infty |\tilde{\eta}(\sigma)|\, d\sigma
  \,,
\end{equation}
and consequently
\begin{equation}\label{reqo}
  \int_{\zeta}^\infty |\tilde{\eta}(\sigma)|\, d\sigma \leq  
 \int_{\zeta}^\infty |f(\sigma)| \, d\sigma \,.
\end{equation}
 Using \eqref{KSlem1e2o} and
 \eqref{reqo},  we obtain for $\zeta \geq  \zeta _0^*$,
\begin{align}\label{bdb}
|\tilde{\eta}(\zeta)| 
&\leq \|G_N\|_{L^\infty} \int_\zeta^\infty |\tilde{g}(\sigma)| \, d\sigma \leq 
 \varepsilon C_N  \int_\zeta^\infty f(\sigma) + |\tilde{\eta}(\sigma)| \, d\sigma \leq 
 \varepsilon C_N  \int_\zeta^\infty f(\sigma) \, d\sigma \leq  \varepsilon C_{N} f(\zeta) \,.
\end{align}
By making $\varepsilon > 0$ smaller if necessary such that $\varepsilon \leq P_N(1 + \varepsilon_0)/C_{N}$ we obtain $\zeta_2^* \leq \zeta_0^*$, and the claim follows. 
\end{proof}

\noindent
{\bf Step 6:} For any small $\varepsilon_0 > 0$, there exists $m_0 > 0$ such that for each $m \geq m_0$ we have $\eta \leq 0$ on $[\zeta_2^*, \infty)$ where $\zeta_2^\ast$ is defined in \eqref{defzeta2ast}.

\begin{proof}[Proof of Step 6] 
Suppose first that $3\leq N \leq 9$. Then, we rewrite \eqref{eqdf} as
\begin{equation}
\label{KSdefetainto0}
\tilde{\eta} (\zeta )= \int^\infty_{\zeta} G_N(\sigma - \zeta) \tilde{g}(\sigma )  d\sigma =: 
 \int^\infty_{\zeta} F(\zeta, \sigma)  d\sigma 
\end{equation}
and
\begin{align*}
\int^\infty_{\zeta} F(\zeta, \sigma)  d\sigma &= \sum_{k = 0}^\infty \int_{\zeta+\frac{2k\pi}{\beta}}^{\zeta+\frac{(2k+1)\pi}{\beta}} F(\zeta, \sigma) d\sigma + \int_{\zeta +\frac{(2k+1)\pi}{\beta}}^{\zeta + \frac{(2k+2)\pi}{\beta}} F(\zeta, \sigma) d\sigma   
\\
&=  \sum_{k = 0}^\infty \int_{\zeta+\frac{2k\pi}{\beta}}^{\zeta+\frac{(2k+1)\pi}{\beta}} F(\zeta, \sigma) + F\left(\zeta, \sigma + \frac{\pi}{\beta}\right) d\sigma  \,,
\end{align*}
where
\begin{align}\label{frl}
F(\zeta, \sigma) + F\left(\zeta, \sigma + \frac{\pi}{\beta}\right) = G_N(\sigma - \zeta) \left(\tilde{g}(\sigma) -   e^{-\frac{\alpha\pi}{2\beta}} \tilde{g}\left(\sigma + \frac{\pi}{\beta}\right)\right) \,.
\end{align}
Recall, for any $\sigma \geq \zeta_2^*$ we have 
$|\eta(\sigma)| \leq P_{N,\varepsilon_0} f(\sigma)$, with $1 >  P_{N,\varepsilon_0}$ for any sufficiently small $\varepsilon_0 > 0$. 
In the following, we use the notation $O(m^{-1})$ for quantities converging to zero as $m \to \infty$. Then, 
 since $\phi$ is decreasing on $(0, \infty)$ and $f \pm \tilde{\eta} \geq 0$ on $[\zeta_2^*, \infty)$, one has 
\begin{multline}
\phi((f + \tilde{\eta})(\sigma)) - e^{-\frac{\alpha\pi}{2\beta}} \phi \left( (f + \tilde{\eta}) \left( \sigma+ \frac{\pi}{\beta} \right) \right)
 \leq \phi((f - |\tilde{\eta}|)(\sigma)) - e^{-\frac{\alpha\pi}{2\beta}} \phi \left( (f + |\tilde{\eta}|) \left( \sigma+ \frac{\pi}{\beta} \right) \right)  
\\
\begin{aligned}
&\leq \phi((1 - P_{N,\varepsilon_0}) f(\sigma)) - e^{-\frac{\alpha\pi}{2\beta}} \phi \left((1 + P_{N,\varepsilon_0})f \left( \sigma+ \frac{\pi}{\beta} \right) \right)  
\\
&\leq  \phi((1 - P_{N,\varepsilon_0}) f(\sigma)) - e^{-\frac{\alpha\pi}{2\beta}} \phi \left((1 + P_{N,\varepsilon_0}) (e^{-\frac{2\pi}{\beta}} + O(m^{-1})) f \left( \sigma \right) \right),
\end{aligned}
\end{multline}
where in the last step we used that, for $\sigma \geq \zeta_2^*$,
\begin{align}\label{fme}
f(\sigma + \pi/\beta) &= e^{-\frac{2\pi}{\beta}} f \left( \sigma \right) + \frac{\pi m^2}{2\beta(N-1)} e^{-\frac{2\pi}{\beta}} e^{-2\sigma}
\leq e^{-\frac{2\pi}{\beta}} f \left( \sigma \right) + \frac{\pi}{\beta \zeta_1} f(\sigma) \\
&\leq 
(e^{-\frac{2\pi}{\beta}} + O(m^{-1})) f \left( \sigma \right) \,.
\end{align}
We claim that  for any sufficiently small $\varepsilon_0, \varepsilon_1 > 0$ and any sufficiently large $m$, one has
\begin{equation}\label{acnd}
\phi \left( (1 -P_{N,\varepsilon_0}) z \right) \leq e^{-\frac{\alpha \pi}{2\beta}} \phi \left( (1 + P_{N,\varepsilon_0}) (e^{-\frac{2\pi}{\beta}}  + O(m^{-1}))z \right) 
- \varepsilon_1 z^2 , \qquad \textrm{for any } z \in \left[0, \Gamma \right] \,.
\end{equation}
Indeed, for any $\kappa > 1$ sufficiently close to one (see below), define 
\begin{equation}
\psi_\kappa(z) = \phi \left( (1 -P_{N}) z \right) - e^{-\frac{\alpha \pi}{2\beta}} \phi \left( (1 + P_{N}) \kappa e^{-\frac{2\pi}{\beta}}  z \right) 
\end{equation}
and note that $\psi_\kappa(0) = \psi'_\kappa(0) = 0$. Moreover, using that $\phi''(z) = -2(N-2) e^z$ and $P_N < 1/3$, we have, for $\kappa = 1$ and any $z \in [0, \Gamma]$,
\begin{align}
\psi''_1(z) &= -2(N-2) \left( (1 - P_{N})^2 e^{(1 -P_{N}) z} - e^{-(\frac{\alpha}{2} + 4) \frac{\pi}{\beta}} (1 + P_{N})^2 e^{(1 + P_{N}) e^{-\frac{2\pi}{\beta}} z} \right)  \\
&\leq -2(N-2) e^{(1 -P_{N}) z}  \left(  (1 - P_{N})^2  - e^{-(\frac{\alpha}{2} + 4) \frac{\pi}{\beta}} (1 + P_{N})^2 e^{ 2 P_{N}z}  \right) \\
&\leq -2(N-2) e^{(1 -P_{N}) z}  \left(  (1 - P_{N})^2  - e^{-(\frac{\alpha}{2} + 4) \frac{\pi}{\beta}} (1 + P_{N})^2 e^{ 2 P_{N}\Gamma}  \right) \\
&\leq -2(N-2) e^{(1 -P_{N}) z}  \left(  \frac{4}{9}  - e^{-\pi} \frac{16}{9} e^{ 2/3 \cdot 1.1}  \right) = - 2c_N < 0.
\end{align}
Then by the continuity $\psi''_\kappa(z) < -c_N < 0$ for any $\kappa > 1$ sufficiently close to $1$ and any $z \in [0, \Gamma]$. Fix such $\kappa_0 > 1$. Thus $\psi_\kappa (z) < -c_N z^2$ on $[0, \Gamma]$, and we 
obtain that \eqref{acnd} holds true for any $\varepsilon_ 1 < c_N$,  for any sufficiently small $\varepsilon_0 > 0$ and large $m$. The claim follows.

 In addition, using that $f$ is decreasing and that $|\tilde{\eta} (\sigma )|\leq P_{N,\varepsilon_0}f(\sigma)$, we have, for $\sigma \geq \zeta_2^* \geq \zeta_1^*$, 
\begin{align*}
m^2 e^{-2\sigma} \left( (\tilde{\eta} +f)(\sigma ) - e^{-\frac{\pi}{2\beta}(4+\alpha) } (\tilde{\eta} +f) \left(\sigma +\dfrac{\pi}{\beta}\right)  \right) 
&\leq 
2 (1+P_{N,\varepsilon_0}) m^2 e^{-2\sigma} f (\sigma)  \\
&\leq \frac{C_N}{\zeta_1^*} f^2(\sigma) \leq \varepsilon_1 f^2(\sigma) \,,
\end{align*}
where we used that by Remark \ref{rmk:zti}, $\zeta_1^* \to \infty$ as $m \to \infty$. Therefore 
recalling that $\tilde{g}(\zeta )= \phi (\eta(\zeta )+m^2 e^{-2\zeta} \eta(\zeta )$, \eqref{acnd} yields
\begin{equation}\label{ginq}
\tilde{g}(\sigma) -   e^{-\frac{\alpha\pi}{2\beta}} \tilde{g}\left(\sigma + \frac{\pi}{\beta}\right) \leq 0.
\end{equation}
Since for any integer $k \geq 0$, one has $G_N(\sigma -\zeta) \geq 0$ on the interval $\left(\zeta+\frac{2k\pi}{\beta},  \zeta+\frac{(2k+1)\pi}{\beta} \right)$, we obtain from 
\eqref{frl}
$$
F(\zeta, \sigma) + F\left(\zeta, \sigma + \frac{\pi}{\beta}\right)  \leq 0,
$$
and Step 6 follows. 

Next, assume $N \geq 10$ and notice that $G_N \geq 0$ in this case.  Also, since $|\tilde{\eta}(\sigma)| \leq P_{N, \varepsilon_0} f(\sigma)$ on $[ \zeta_2^* ,\infty)$ and 
$P_{N, \varepsilon_ 0} < 1$ for any sufficiently small $\varepsilon_0$, we obtain that $\eta = f + \tilde{\eta} \geq 0$ on $[\zeta_2^*, \infty)$. Since $e^{x}-1-x\geq \dfrac{1}{2}x^2$ for $x \geq 0$, then for any $\zeta \geq \zeta_2^*$,
\begin{align*}
\tilde{\eta} (\zeta )\leq  \int^\infty_{\zeta} G_N(\sigma - \zeta) \Big(m^2 e^{-2\sigma} \eta(\sigma )- (N-2) \eta^2 (\sigma ) \Big)  d\sigma .
\end{align*}
Also, since $\eta \geq 0$ we have 
\begin{align}
m^2 e^{-2\sigma} \eta(\sigma )- (N-2) \eta^2 (\sigma ) &\leq \eta(\sigma ) \left(c_N \frac{f(\sigma)}{\sigma} - (f(\sigma) - |\tilde{\eta}(\sigma)|) 
\right)
\\
&\leq \eta(\sigma ) f(\sigma) \left( \frac{c_N}{\zeta_1^*} - (1 - P_{N, \varepsilon_0}) \right) \leq 0 \,,
\end{align}
where $c_N$ depends only on $N$ and the last inequality follows for any sufficiently large $\zeta_1^*$, that is, for sufficiently large $m$.
Thus $\eta(\zeta) \leq 0$ for each $\tilde{\eta} \geq \zeta_2^*$ as desired.  
\end{proof}

{\bf Step 7:} For any sufficiently small $\varepsilon_0 > 0$, there exists $m_0$ such that for each $m \geq m_0$ we have 
$\zeta_2^* = \zeta_1^*$, where $\zeta_2^\ast$ is defined in \eqref{defzeta2ast} and $\zeta_1^\ast$ is the largest solution to $f(\zeta )=\Gamma$. In particular, 
$|\tilde{\eta} (\zeta_\lambda )|\leq \tilde{C}_N$.

\begin{proof}[Proof of Step 7] 
In Step $6$, we proved that $\tilde{\eta} \leq 0$ on $(\zeta_2^*, \infty)$. In order to obtain an estimate on $|\tilde{\eta}|$, we need a lower bound on $\tilde{\eta}$. 

First assume $3\leq N \leq 9$. Since $ G_N(\sigma - \zeta) \leq 0$  on the interval $\left(\zeta+\frac{(2k+1)\pi}{\beta}, 
\zeta+\frac{(2k+2)\pi}{\beta}\right)$,   
\eqref{ginq} and \eqref{frl} yield 
$$
F(\zeta, \sigma) + F\left(\zeta, \sigma + \frac{\pi}{\beta}\right)  \geq 0.
$$ 
Consequently by using that $\phi$ is decreasing and $\tilde{\eta} \leq 0$, we obtain, for any $\zeta \geq \zeta_2^*$,
\begin{align}
\tilde{\eta}(\zeta) &\geq
 \int_{\zeta}^{\zeta + \frac{\pi}{\beta}} G_N(\sigma - \zeta)\phi((f + \tilde{\eta})(\sigma)) d\sigma 
 +  m^2 \int_{\zeta}^{\infty} G_N(\sigma - \zeta) e^{-2\sigma} (\tilde{\eta} +f) (\sigma ) d\sigma \\
&\geq  \int_{\zeta}^{\zeta + \frac{\pi}{\beta}} G_N(\sigma - \zeta) \phi(f(\sigma)) d\sigma  - m^2 (1+P_{N, \varepsilon_0})  \int_{\zeta}^{\infty} |G_N(\sigma - \zeta)| e^{-2\sigma} f (\sigma ) d\sigma .
\end{align}
In order to estimate $\phi(f(\sigma))$ we use that, for any $y \geq x > 0$, one has 
\begin{equation}
\frac{\phi(x)}{\phi(y)} \leq \frac{x^2}{y^2} \,.
\end{equation}
Indeed this inequality is equivalent to 
\begin{equation}\label{einq}
\frac{e^x - x - 1}{x^2} \leq \frac{e^y - y - 1}{y^2}
\end{equation}
which is true since the function $x \mapsto (e^x - x - 1)/x^2$ is increasing on $(0, \infty)$. Hence, since $\phi < 0$ on $(0, \infty)$ we have 
\begin{equation}
\phi(f(\sigma)) \geq \phi(f(\zeta)) \left( \frac{f(\sigma)}{f(\zeta)} \right)^2 = \phi(f(\zeta)) e^{-4(\sigma - \zeta)} \left( \frac{\sigma + c_N}{\zeta + c_N} \right)^2 \,.
\end{equation}
Using that $\sigma \in (\zeta, \zeta + \pi/\beta)$ and $\zeta \geq \zeta_1^* \to \infty$  as $m \to \infty$, we have 
\begin{equation}\label{oldi}
\phi(f(\sigma)) \geq \phi(f(\zeta)) e^{-4(\sigma - \zeta)} (1 + O(m^{-1})) \,,
\end{equation}
and therefore
\begin{equation}
m^2 (1+P_{N, \varepsilon_0})  \int_{\zeta}^{\infty} |G_N(\sigma - \zeta)| e^{-2\sigma} f (\sigma ) d\sigma 
\leq c_N m^2 f(\zeta) e^{-2\zeta} \leq  c_N\frac{f^2(\zeta)}{\zeta} = O(m^{-1}) f^2(\zeta)\,.
\end{equation}
Thus, for any $\zeta \geq \zeta_2^*$
\begin{align}
\tilde{\eta}(\zeta)
&\geq \dfrac{\phi(f(\zeta))}{\beta} \int_{\zeta}^{\zeta + \frac{\pi}{\beta}} e^{-(\frac{\alpha}{2} +4)(\sigma - \zeta)} \sin(\beta(\sigma - \zeta)) d\sigma - O(m^{-1}) f^2(\zeta)
\\
&= \dfrac{4\phi(f(\zeta))}{(\alpha + 8)^2 + 4\beta^2} (1 + e^{-\frac{(\alpha + 8) \pi}{2\beta}}) - O(m^{-1})f^2(\zeta)  
  \,.
\end{align}
Using again that  $\tilde{\eta} \leq 0$ and  $x \mapsto \phi(x)/x$, is decreasing we obtain for any $\zeta \geq \zeta_2^* \geq \zeta_1^*$, that is, $f(\zeta) \in (0, \Gamma]$ 
and sufficiently large $m$
\begin{align}
|\tilde{\eta}(\zeta)| &\leq 
  \dfrac{4\phi(f(\zeta))}{((\alpha + 8)^2 + 4\beta^2)f(\zeta)} (1 + e^{-\frac{(\alpha +8)\pi}{2\beta}}) f(\zeta) +  O(m^{-1}) f^2(\zeta)
\\  
  &\leq
\left( P_N + O(m^{-1}) \right) f(\zeta) < \left(1 + \frac{\varepsilon_0}{2}\right) P_N f(\zeta) \,.
\end{align}
If $\zeta_2^* > \zeta_1^*$ , then, by continuity, $|\tilde{\eta}(\zeta)| \leq (1 + \varepsilon_0)P_N |f(\zeta)| $  holds for any $\zeta_1^* \leq \zeta \leq \zeta_2^*$ sufficiently close to $\zeta_2^*$, a contradiction to the definition of $\zeta_2^*$. Thus 
$\zeta_1^* = \zeta_2^*$ as desired. 

If $N \geq 10$,  
using $G_N \geq 0$, the monotonicity of $\phi$, and $\tilde{\eta} \leq 0$ as above, we obtain, for any $\zeta \geq \zeta_1^*$,
\begin{align*}
\eta (\zeta )&\geq  \int^\infty_{\zeta} G_N(\sigma - \zeta)(m^2 e^{-2\sigma} \tilde{\eta}(\sigma )+\phi (f(\sigma )) )  d\sigma 
\\
&\geq \frac{\phi (f(\zeta))}{(\zeta + c_N)^2} \int^\infty_{\zeta} G_N(\sigma - \zeta)e^{-4(\sigma - \zeta)}(\sigma + c_N)^2  d\sigma - O(m^{-1}) f^2(\zeta) \,,
\end{align*}
and note that we could not use \eqref{oldi} since $\sigma - \zeta$ is unbounded. 
Then, if $N > 10$, one has 
\begin{equation}
\tilde{\eta} (\zeta) \geq  \dfrac{\phi (f(\zeta))}{( (\alpha /2 + 4)^2 - \beta^2 )} \left(1 - \frac{\tilde{c}_N}{\zeta + c_N} \right)   -   O(m^{-1}) f^2(\zeta ) \,
\end{equation}
and using again that $|\phi(f(\zeta))| \leq \bar{c}_N f^2(\zeta)$ for any $\zeta \geq \zeta_1^*$ and $\zeta \to \infty$ as $m \to \infty$, we have 
\begin{equation}
\tilde{\eta} (\zeta) \geq  \dfrac{\phi (f(\zeta))}{( (\alpha /2 + 4)^2 - \beta^2 )}  -   O(m^{-1}) f^2(\zeta ) \,.
\end{equation}
If $N = 10$, one similarly has 
\begin{equation}
\tilde{\eta} (\zeta) \geq  \dfrac{\phi (f(\zeta))}{ (\alpha /2 + 4)^2} -   O(m^{-1}) f^2(\zeta ) \,.
\end{equation}
The rest of the proof is the same as in the case $3 \leq N \leq 9$. 
\end{proof}

\begin{rmq} \label{pbou}
In Steps 4-7 we proved that 
\begin{equation}
0 \geq \tilde{\eta}(\zeta) \geq -f(\zeta) \geq -\Gamma \qquad \textrm{for any } \zeta \geq \zeta^*_1 
\end{equation}
which in turn implies 
\begin{equation}
0 \leq \eta \leq f(\zeta) \qquad \textrm{for any } \zeta \geq \zeta^*_1 \,.
\end{equation}
In the original variables, for $U^\ast_\lambda (r) = \eta(\zeta) + 2 \zeta$ we have
\begin{equation}
-2\ln r + \ln \frac{2(N - 2)}{\lambda} \leq   U^\ast_\lambda(r) \leq -2\ln r + \ln \frac{2(N - 2)}{\lambda}  + c_N r^2(1 - \ln r)
\qquad \textrm{for any } r \leq \tilde{c}_N \,.  
\end{equation}
The importance of this bound is in the estimate on $U^*_\lambda$ on an interval which is independent of $\lambda$. An interested 
reader can calculate explicitly constants $c_N$ and $\tilde{c}_N$.
\end{rmq}

\begin{rmq}\label{deub}
From Remark \ref{pbou} we can also obtain an estimates on $(U^\ast_\lambda)'$ as follows. By \eqref{eqdf}
\begin{equation}
\tilde{\eta}'(\zeta) = -\int_\zeta^\infty G_N'(\sigma - \zeta) \tilde{g}(\sigma) \, d\sigma \,.
\end{equation}
Since $G'_N$ is a bounded integrable function, using Remark \ref{pbou} and analogous estimates as in 
\eqref{bdb} we have 
\begin{equation}
|\tilde{\eta}'(\zeta)| \leq C_N f(\zeta) \qquad \textrm{for any } \zeta \geq \zeta^*_1 \,.
\end{equation}
In the original variables the last bound translates into
\begin{multline}
\left|(U^\ast_\lambda)'(r) + \frac{2}{r} +  r c_N^* - \frac{r}{N - 1} \left(\ln \sqrt{\frac{2(N-2)}{\lambda}} - \ln r \right)\right| = \frac{1}{r}|\tilde{\eta}(\zeta)| \\
\leq  
C_N\frac{1}{r} f(\zeta) \leq C_{N, \lambda} r (1 + |\ln r|) \qquad \textrm{for any } r \leq c_N \,,
\end{multline}
where $C_{N, \lambda}$ is bounded in $\lambda$ uniformly on compact subsets of $(0, \infty)$.
\end{rmq}

\noindent {\bf Step 8:} Proof of Proposition \ref{KSpinfty}.
\medskip

\noindent\textit{Proof of Step 8.}
Recall that $r_\lambda$ is the smallest solution of $U^*(r) = \overline{u}_\lambda$ and 
$\zeta_\lambda$ is the corresponding transformed variable, see \eqref{relrzeta}. 
Denote $z_\lambda := M_\lambda^1$, that is, $z_\lambda$ is 
first critical point of $U^\ast$ and let $\rho_\lambda$ be its transformed variable, see \eqref{relrzeta} with $r_\lambda$ and $\zeta_\lambda$ replaced respectively 
by $z_\lambda$ and $\rho_\lambda$.

First, we show that $z_\lambda \geq r_\lambda$. Indeed, otherwise $z_\lambda < r_\lambda$ and as in the proof of 
Lemma \ref{liouville}, we have that 
the function $V$ defined by \eqref{dflf} is decreasing in $r$.  Then, as in  the proof of Lemma \ref{liouville} 
we obtain that $U^*(r) \geq U^*(z_\lambda) > \bar{u}_\lambda$
a contradiction to Lemma \ref{corbehazero}.
Hence for the rest of the proof we assume that $z_\lambda \geq r_\lambda$. 

By Steps 5-7 (cf. Remark \ref{pbou}) one has $|\tilde{\eta} (\zeta )|\leq C_N$ for any $\zeta \geq \zeta^*_1$.  
In particular, 
$|\tilde{\eta} (\zeta_\lambda)| \leq C_N$ if 
$\zeta_\lambda \geq \zeta^*$, that is, if $f(\zeta_\lambda) \leq \Gamma$. But $|\tilde{\eta} (\zeta_\lambda)| \leq C_N$
holds also by Step 4 if $f(\zeta_\lambda) \geq \Gamma$. Overall, we have $|\tilde{\eta} (\zeta_\lambda)| \leq C_N$.

We claim that $|\eta (\zeta )|\leq C_N$ holds in fact for all $\zeta \geq \zeta_\lambda$ and any sufficiently small $\lambda > 0$. Indeed, if $\zeta_\lambda \geq \zeta_1^*$, then
the statement is already proved in Steps 5-7. If $\zeta_\lambda < \zeta_1^*$, 
assume that there exists $\tilde{\zeta} \in (\zeta_\lambda, \zeta_1^*]$ such that $\tilde{\eta}' (\tilde{\zeta}) = 0$. Without loss of generality let 
$\tilde{\zeta}$ be the largest such number.
Since $\tilde{\zeta} < \zeta_1^*$, then $f(\tilde{\zeta}) \geq \Gamma = 1.1$, and consequently for any large $m$ (or small $\lambda$), 
one has $f'(\tilde{\zeta}) < - \frac{2}{\Gamma} f(\tilde{\zeta}) \leq -2$. 
If $\tilde{r}$  corresponds to $\tilde{\zeta}$, see \eqref{relrzeta}, then
$$
(U^\ast)^\prime (\tilde{r}) =- \frac{1}{\tilde{r}} (f'(\tilde{\zeta}) + \tilde{\eta}' (\tilde{\zeta}) + 2) =  - \frac{1}{\tilde{r}} (f'(\tilde{\zeta}) + 2)  > 0.
$$ 
Since $(U^\ast)^\prime (r) < 0$ for $r$ sufficiently close to $0$, we obtain $\tilde{r} \geq z_\lambda$, and therefore $\tilde{\zeta} \leq  \zeta_\lambda$, a contradiction to $\zeta_\lambda < \tilde{\zeta}$. Thus, no such $\tilde{\zeta}$ exists, and therefore
$\tilde{\eta}$ is increasing on $(\zeta_\lambda ,\zeta_1^*)$. Since $\max\{|\eta (\tilde{\eta}_\lambda)|, |\tilde{\eta}(\zeta_1^*)| \} \leq C_N$, 
we deduce that  $|\tilde{\eta} (\zeta )|\leq C_N$, for all $\zeta \geq \zeta_\lambda$ and the claim follows.

Now, $|\tilde{\eta}| \leq C_N$ implies that $|\tilde{g}| \leq C_N$ (defined in \eqref{eq:neq}) on $(\zeta_\lambda, \infty)$. Differentiating \eqref{eqdf} and using that $G'_N$ is integrable, we find
\begin{equation}
|\tilde{\eta}'(\zeta_\lambda)| \leq \int_{\zeta_\lambda}^\infty |G_N'(\sigma - \zeta_\lambda)| |\tilde{g}(\sigma)| d\sigma \leq C_N.
\end{equation}
Then,  using \eqref{relrzeta} and $|f'(\zeta )| \leq 3 f(\zeta)$, for any $|\zeta|$ large enough, we have
\begin{equation}
|(U^\ast )^\prime (r_\lambda)| \leq  \frac{ |f'(\zeta_\lambda)| + |\tilde{\eta}'(\zeta_\lambda)|  + 2}{r_\lambda} \leq \frac{C_N}{r_\lambda}  = C_N\sqrt{\lambda e^{2\zeta_\lambda}} \,. 
\end{equation} 
Furthermore, recalling that $U^\ast (r) = 2\zeta + f(\zeta) + \tilde{\eta}(\zeta)$, we deduce 
from Step 3, and $|\tilde{\eta} (\zeta_\lambda)| \leq C_N$ that  $2\zeta_\lambda - c_N \leq U^\ast (r_\lambda) = \overline{u}_\lambda \leq 2\zeta_\lambda + c_N$. Consequently,  $\overline{u}_\lambda = \lambda e^{\overline{u}_\lambda}$ yields
\begin{equation}
|(U^\ast)^\prime (r_\lambda)| \leq  C_N \sqrt{ \lambda e^{\overline{u}_\lambda}} = C_N \sqrt{\overline{u}_\lambda}.
\end{equation}

Recalling that $v := U^\ast /\overline{u}_\lambda$ solves \eqref{eq:vag}, we obtain
\begin{equation}
\label{KSvprime}
|v'(r_\lambda)| \leq C_N \frac{1}{\sqrt{\overline{u}_\lambda}} \qquad \textrm{and} \quad v(r_\lambda) = 1 \,.
\end{equation}
 Furthermore, since the function 
\begin{equation}
r \mapsto E(r) = \frac{(v^\prime (r))^2}{2} + \frac{e^{\overline{u}_\lambda (v(r) - 1)}}{\overline{u}_\lambda} - \frac{v^2 (r)}{2}
\end{equation}
is non-increasing,  any $r \geq r_\lambda$ one has
\begin{equation}
\frac{C_N}{\overline{u}_\lambda} - \frac{1}{2} \geq  E(r_\lambda) \geq E(r) \geq \frac{e^{\overline{u}_\lambda (v(r) - 1)}}{\overline{u}_\lambda} - \frac{v^2(r)}{2}
\geq - \frac{v^2(r)}{2} \, .
\end{equation}
 Thus, for sufficiently small $\lambda > 0$, 
\begin{equation}
v^2(r)  \geq 1 - \frac{C_N}{\overline{u}_\lambda}  \qquad \textrm{for any} \quad r \geq r_\lambda \,,
\end{equation}
and consequently 
\begin{equation}\label{npv}
\sup_{r \geq r_\lambda} (1 - v)_+ \to 0 \qquad \textrm{as} \quad \lambda \to 0 \,,
\end{equation}
 where $g_+ = \max\{g, 0\}$ denotes a positive part of a function $g$. Recall that $w(r)=r^{\frac{N-1}{2}}(v (r) - 1)$ satisfies \eqref{eq:wag} and clearly 
\begin{equation}
\dfrac{ e^{\overline{u}_\lambda (v -1)} -v}{v -1 } = \dfrac{ e^{\overline{u}_\lambda (v -1)} - 1}{v -1 } - 1 \,.
\end{equation} 
Fix  any $\mu > 0$ and $a > 0$ and denote $I_a := \left[\frac{a}{4}, a\right]$.
Choose any $r \in I_a$. If $\overline{u}_\lambda (v(r) -1)  \geq -1$, then using that $x \mapsto (e^x - 1)/x$ is increasing  we have for sufficiently small $\lambda$ (or large $\bar{u}_\lambda$ by
Step 1)
\begin{equation}
 \dfrac{ e^{\overline{u}_{\lambda } (v(r)  -1)} - 1}{v(r)  - 1 } =  \dfrac{ e^{\overline{u}_{\lambda } (v(r)  -1)} - 1}{\overline{u}_{\lambda} (v(r)  -1)} \overline{u}_\lambda  
 \geq   \overline{u}_\lambda \left(1 - \frac{1}{e}\right) \geq \mu \,. 
\end{equation}
On the other hand if $\overline{u}_\lambda (v(r) -1)  < -1$, then  $v(r) < 1$ and 
\begin{equation}
 \dfrac{ e^{\overline{u}_{\lambda } (v(r)  -1)} - 1}{v(r)  -1 } \geq \dfrac{e^{-1} - 1}{v(r)  - 1} \geq \frac{1 - e^{-1}}{\sup_{\rho \geq r_\lambda}(1 - v(\rho))_+} >  \mu \,, 
\end{equation}
for sufficiently small $\lambda$, where we used \eqref{npv} in the last inequality. 
Hence, for any $\mu > 0$ and $a > 0$ one has for sufficiently small $\lambda > 0$ that 
$$ 
\dfrac{ e^{\overline{u}_\lambda (v -1)} -v}{v -1 }-\dfrac{(N-1)(N-3)}{4r^2}  \geq \mu - C_{N, a} \qquad \textrm{for any } \quad r \in I_a := \left[\frac{a}{4}, a\right].
$$
Thus, given $a > 0$ and integer $i > 0$,  there is large $\mu$, 
such that a solution of  the  equation $z'' + (\mu - C_{N, a} )z = 0$ has at least $i + 2$ zeros on $I_a$, and by 
Sturm-Picone  oscillation theorem for any sufficiently small $\lambda > 0$, the function $w$ has at least $i + 1$ zeros on $I_a$. 
Consequently, $U^* (r) = \overline{u}_\lambda$ has at least $i + 1$ solutions on  $I_a$, and therefore $U^*$ has at least $i$
critical points on $I_a$. In a different notation for any $j \in \{1, \cdots, i\}$ and any $a > 0$ one has $R^j_\lambda < a$ for any sufficiently small $\lambda > 0$.   
\end{proof}

\begin{prop}
\label{KScontmp}
For any $i\in \N$, the function $\lambda \rightarrow R_\lambda^i$ is continuous.
\end{prop}

\begin{proof}
Fix $\lambda^\ast>0$ and an open interval $I_0 = (A, B)$ with $0 < A < B < \infty$. Then by Remark \ref{pbou} there is $\delta > 0$ such that for any 
$\lambda \in (\lambda^\ast - \delta, \lambda^* + \delta)$ one has
\begin{equation}
|U^*_\lambda(A)| \leq C_N \,.
\end{equation}
If for some $\lambda \in (\lambda^\ast - \delta, \lambda^* + \delta)$ the function $U^*_\lambda$ is decreasing on $(A, B)$, then
non-negativity of $U^*$ yields $|U^*_\lambda| \leq C_N$ on $(A, B)$. If for some  $\lambda \in (\lambda^\ast - \delta, \lambda^* + \delta)$,
there is a smallest $z_\lambda < B$ such that $(U^*_\lambda)'(z_\lambda) = 0$, then since $V$ defined by \eqref{dflf} 
is non-decreasing and by the Step 8 in 
the proof of Proposition \ref{KSpinfty} one has $U^*_\lambda(z_\lambda) < \bar{u}_\lambda$, then, for any $r \geq z_\lambda$,
\begin{equation}
  \bar{u}_\lambda \geq
  \lambda e^{U^*_\lambda (z_\lambda)} \geq
  \lambda e^{U^*_\lambda (z_\lambda)}  -  \frac{1}{2} (U^*_\lambda (z_\lambda))^2  = V(r_\lambda) \geq V(r) \geq 
 \lambda e^{U^*_\lambda (r)}  - \frac{1}{2}  (U^*_\lambda (r))^2 .
\end{equation}
Since $\lambda \in (\lambda^\ast - \delta, \lambda^* + \delta)$, we obtain that the left hand side is bounded by a constant independent of 
$\lambda$, and consequently $U^*_\lambda $ is bounded on $(z_\lambda, \infty)$, by a constant independent of $\lambda \in (\lambda^\ast - \delta, \lambda^* + \delta)$. Overall, we have 
\begin{equation*}
\sup_{\lambda \in (\lambda^* - \delta, \lambda^* + \delta)} \sup_{(A, B)} U_\lambda^\ast \leq C(A) \,.
\end{equation*}
Then, the elliptic regularity implies  that, for any $q > 1$,
\begin{equation}
\label{KSconte1}
\|U^\ast_\lambda \|_{W^{3, q}(I_0)} \leq C(N, q, A, B - A,  \lambda^\ast, \delta) \qquad \textrm{for any } \lambda \in (\lambda^* - \delta, \lambda^* + \delta) \,.
\end{equation}

For $\alpha_0 \in (0, 1)$, we choose $q_0>0$ large enough such that $W^{3, q_0} (I_0) \hookrightarrow C^{2 + \alpha_0}(I_0)$.  Let $(\lambda_n)_{n \in \N}$ 
be a sequence such that $\lambda_n \to \lambda^*$ when $n\rightarrow \infty$. Thanks to \eqref{KSconte1}, using Arzel\` a-Ascoli's theorem, there exists a subsequence, 
still denoted $(\lambda_n)$, such that $U^\ast_{\lambda_n} \to w$, as $n\to \infty$, in $C^{2}(I_0)$. Noticing that
\begin{equation*}
|\lambda_n e^{U^\ast_{\lambda_n} (s)} -\lambda^\ast e^{ w(s)}| \leq \lambda_n | e^{U^\ast_{\lambda_n} (s)} - e^{ w(s)}| + |\lambda_n -\lambda^\ast| e^{w(s)}
\to 0 \qquad \textrm{as } n \to \infty,
\end{equation*}
we deduce that $w$ satisfies the equation
\begin{equation*}
- \Delta w + w =\lambda^\ast e^{ w}, \qquad \textrm{in } I_0 \,.
\end{equation*}
Since $I_0$ is an arbitrary compact interval, proceeding as above and using standard diagonal arguments, 
we obtain the existence of a subsequence $(\lambda_{k_n})_n$, $\lambda_n \in (\lambda^\ast -\delta , \lambda^\ast +\delta )$, for all $n\in \N$, such that $U^\ast_{\lambda_{k_n}} \to w$, as $n\to \infty$, in $C^{2}_{loc}((0,\infty ))$, for some function $w$ satisfying
\begin{equation*}
- \Delta w + w =\lambda^\ast e^{w} \qquad \textrm{in } (0, \infty) \,.
\end{equation*} 
Next, we claim that $w$ is in fact $U_{\lambda^\ast}^\ast$. Using the uniqueness of solution proved in Proposition \ref{uniq}, it is sufficient to show that 
\begin{equation}
\label{KSconte2}
\lim_{r\rightarrow 0^+}  w(r)+ 2\ln r=A_{\lambda^\ast  ,N},
\end{equation}
where $A_{\lambda ,N}=\ln  \dfrac{2(N-2)}{\lambda}$. However, by Remark \ref{pbou} there is $r_0(\varepsilon)$ independent of
$ \lambda \in (\lambda^\ast - \delta, \lambda^* + \delta)$ such that 
\begin{equation*}
 A_{\lambda, N} \leq U_\lambda^\ast (r)+2\ln r \leq  A_{\lambda, N} + \varepsilon \qquad \textrm{for all }\quad r \in (0, r_0(\varepsilon)) \,.
\end{equation*} 
Clearly $A_{\lambda_n,N}\rightarrow A_{\lambda^\ast ,N}$ when $n\to \infty$ and
using that $U_{\lambda_n}^\ast \to w$ in 
$C^2_{\textrm{loc}} ((0, r_0(\varepsilon))) $, we obtain
\begin{equation*}
A_{\lambda^*, N} \leq w(r) +2\ln r\leq  A_{\lambda^*, N} + \varepsilon , \qquad \textrm{for all }\quad r \in (0, r_0(\varepsilon)) \,.
\end{equation*} 
Since $\varepsilon > 0$ is arbitrary, we conclude that \eqref{KSconte2} holds, and therefore $w=U_{\lambda^\ast}^\ast$
by the uniqueness. Hence, 
\begin{equation}
\label{KSconte3}
U_\lambda^\ast \rightarrow U_{\lambda^\ast}^\ast ,\textrm{ as }\ \lambda\to \lambda^\ast,\ \textrm{ in }\ C_{loc}^2((0,\infty)).
\end{equation}
Finally, we prove the continuity of the function $\lambda\rightarrow R_\lambda^i$. In the following, we assume that $R_\lambda^i$ is a local minimum of $U_\lambda^\ast$, 
the case of local maximum follows analogously. Note that $U^*(R_\lambda^i) \neq \bar{u}_\lambda$, 
otherwise $U^* \equiv \bar{u}_\lambda$, and 
we have a contradiction to the uniqueness of the initial value problem. Thus, for any sufficiently small $\bar{\varepsilon} > 0$, we get
\begin{equation*}
U^\ast_{\lambda^\ast}(R_{\lambda^*}^i - \bar{\varepsilon}) > U^\ast_{\lambda^\ast}(R_{\lambda^*}^i) \qquad \textrm{and } \quad U^\ast_{\lambda^\ast}(R_{\lambda^*}^i + \bar{\varepsilon}) > U^\ast_{\lambda^\ast}(R_{\lambda^*}^i) \,.
\end{equation*}
Then \eqref{KSconte3}, yields that for sufficiently small $\lambda > 0$ there exists a local minimizer $q_\lambda$ of $U^\ast_{\lambda}$ 
in any small neighborhood of $R_{\lambda^*}^i$, or equivalently for every $\lambda > 0$ there is $q_\lambda$ with $(U^\ast_{\lambda})'(q_\lambda) = 0$
such that
\begin{equation}
\lim_{\lambda \to 0} q_\lambda =  R_{\lambda^*}^i \,.
\end{equation}
On the other hand assume that there exists a sequence $(\lambda_n)_{n \in \N}$ such that $\lambda_n \to \lambda^*$ and $(q_{\lambda_n})_{n \in \N}$
converges to $q^*$. Then by \eqref{KSconte3} one has $(U^*_{\lambda^*})' (q^*)  = 0$.  

Finally assume that there exists a sequence $(\lambda_n)_{n \in \N}$ such that $\lambda_n \to \lambda^*$ and both sequences $(q_{\lambda_n})_{n \in \N}$, 
$(q_{\lambda_n}')_{n \in \N}$ converges to $q^*$. Then by the mean value theorem, there exists $s_{\lambda_n}$ between $q_{\lambda_n}$ and 
$q_{\lambda_n}'$ such that $(U^\ast)''(s_{\lambda_n}) = 0$. By passing to the limit, one has $(U^*_{\lambda^*})' (q^*)  = (U^*_{\lambda^*})'' (q^*)  = 0$, a contradiction to the fact that every critical point is either strict minmizer or strict maximizer. 

Overall, we proved that in each neighborhood of $R^i_{\lambda^*}$, for sufficiently small $\lambda$, there exists exactly one critical point of 
$U^*_\lambda$ and the proof is finished. 
\end{proof}

\begin{proof}[Proof of Theorem \ref{KSmainthm}.]  By the definition of $i^*$, for any $i\geq \tilde{i}$ we have $R_{\tilde{\lambda}}^{i}>R$. On the other hand, by Proposition \ref{KSpinfty}, we know that, for any $i > 0$, $\lim_{\lambda\rightarrow 0}R_\lambda^i < R$. Since for any $i\in \N$ the function $\lambda\rightarrow R^i_\lambda$
is continuous by Proposition \ref{KScontmp}, we deduce that there exists $\lambda^i > \tilde{\lambda}$ such that $R_{\lambda^i}^i =R$. This concludes the proof.
\end{proof}

\section{Oscillation of the branches for generic radius}\label{sec:gen}
In this section, we prove two generic uniqueness results, one for singular and one for regular solutions 
i.e. we prove Theorems \ref{generic} and \ref{generic2}. More precisely we show that, for generic $R>0$, if $(U^*_{\lambda^*})'(R) = 0$
for some $\lambda^* > 0$, then $(U^*_{\lambda})'(R) \neq 0$, for any $\lambda \approx \lambda^\ast$, $\lambda \neq \lambda^\ast$ and 
that there exists at most one $\lambda$ such that $r^i_{\lambda ,\gamma}=R$, if $\lambda \approx \lambda^\ast$ and 
$\gamma$ is large enough. The proof of Theorem \ref{generic} relies on the Sard's theorem applied to the function $\lambda \rightarrow R_\lambda^i$, and therefore 
our first goal is to show that this function is Lipschitz. We start with the following lemma.

\begin{lem}\label{lrgl}
For any $\lambda^* > 0$ and any bounded $I \subset (0, \infty)$, there exists $C_{\lambda^*, I}>0$ locally bounded in $\lambda^*$ such that
\begin{equation}
\|U^*_{\lambda^*} - U^*_\lambda\|_{C^{2}(I)} \leq C_{\lambda^*, I} |\lambda - \lambda^*|,
\end{equation}
for any $\lambda$ sufficiently close to $\lambda^*$.
\end{lem}

\begin{proof}
Fix any $\lambda >0$ and denote $V_\lambda = U_\lambda + \delta^*$, where $\delta^* = \ln(\lambda/\lambda^*)$. Then, $V_\lambda$
satisfies
\begin{equation}
- V_\lambda'' - \frac{N-1}{r} V_\lambda' + V_\lambda  = \lambda^* e^{V_\lambda} + \delta^* 
\end{equation}
and 
\begin{equation}
V_\lambda (r) = -2\ln r + \ln \left(\frac{2(N-2)}{\lambda^*} \right) + O(r^{2 - \delta}) \,.
\end{equation}
Denote $W = U^*_{\lambda^*} - V_\lambda$. We see that $W$ satisfies 
\begin{equation}
- W'' - \frac{N-1}{r} W' + W  = \lambda^* (e^{U^*_{\lambda^*}} -  e^{V_\lambda}  ) - \delta^* =
\lambda^* e^{U^*_{\lambda^*}} (1 -  e^{-W}) - \delta^*
\end{equation}
and after simple algebraic manipulations, we end up with
\begin{equation}
- W'' - \frac{N-1}{r} W' + W - \frac{2(N-2)}{r^2} W = 
\left(\lambda^* e^{U^*_{\lambda^*}} \frac{1 -  e^{-W} - W}{W} + \lambda^* e^{U^*_{\lambda^*}}-\frac{2(N-2)}{r^2} \right) W  - \delta^*
\,.
\end{equation}
Furthermore, 
\begin{equation}
W(r) = O(r^{2 - \delta}) \,,
\end{equation}
where $O(r^{2 - \delta})$ in general depends on $\lambda$ and $\lambda^*$. 
By Remark \ref{pbou} one has that $|W| \leq 1$ on $(0, c_N)$ and combined with \eqref{KSconte3} one has $|W| \leq 1$ on $I$
for any $\lambda$ sufficiently close to $\lambda^*$.
Hence, 
\begin{equation}
\left| \frac{1 -  e^{-W} - W}{W} \right| \leq C W^2 \qquad \textrm{ on } I \,,
\end{equation} 
and then by Remark \ref{pbou} for $\delta < 1$, we have
\begin{equation}
\left|\lambda^* e^{U^*_{\lambda^*}} \frac{1 -  e^{-W} - W}{W}\right| 
\leq c_N \frac{1}{r^2} (r^{2 - \delta})^2 \leq c_N \qquad \textrm{ on } I \,.
\end{equation}
Also, by Remark \ref{pbou}, we infer that
\begin{equation}
\left|\lambda^* e^{U^*_{\lambda^*}}-\frac{2(N-2)}{r^2} \right| = \frac{2(N - 2)}{r^2} \left| e^{q(r)} - 1 \right| 
\,,
\end{equation}
where $|q(r)| \leq c_N r^2(1 + |\ln r|)$. Thus, if we set $\bar{W} = W/\delta^*$, we obtain
\begin{equation}
- \bar{W}'' - \frac{N-1}{r} \bar{W}' - \frac{2(N-2)}{r^2} \bar{W} = m(r) \bar{W} - 1 
\end{equation}
with $|m(r)| \leq c_N(1 + |\ln r|)$. Furthermore, Remark \ref{deub} implies $\bar{W}'(0) = 0$ and this
condition is fulfilled continuously. 
 Finally, denote $Z_1(r) = \bar{W}(r)$, $Z_2(r) = r\bar{W}'(r)$ and $Z = (Z_1, Z_2)$. Then 
\begin{equation}
Z' = \frac{1}{r} J Z - rm(r) 
\left(
\begin{array}{c}
0 \\
Z_1
\end{array}
\right)
 + \left(
\begin{array}{c}
0 \\
r
\end{array}
\right) \,,
 \qquad \qquad Z(0) = 0 \,,
\end{equation}
where 
\begin{equation}
J = \left(
\begin{array}{cc}
0&1\\
-2(N - 2)& 2 - N
\end{array}
\right)\,.
\end{equation}
Since the eigenvalues $\mu_1, \mu_2$ of $J$ have negative real parts, proceeding as in Lemma $2.3$ of \cite{koch2001initial}, one can show that $Z_1$ is bounded on $I$. 

Consequently, $\bar{W}$ is bounded, and therefore $|W| \leq C |\delta^*|$ on $I$. Since 
$|\delta^*| \leq C_{\lambda^*} |\lambda - \lambda^*|$ for any $\lambda$ sufficiently close to $\lambda^*$, the assertion of the lemma follows from standard regularity theory since $I$ is bounded away from the origin,
and therefore the coefficients of the equation are bounded, uniformly in $\lambda \in (\lambda^* - \delta, \lambda^* + \delta)$. 
This also implies that 
$\lambda^* \mapsto C_{\lambda^*}$ is bounded locally uniformly on $(0, \infty)$.
\end{proof}

\begin{lem}\label{lmlip}
The function $\lambda \mapsto R^i_\lambda$ is Lipschitz.
\end{lem}

\begin{proof}
Fix $\lambda^*>0$ and $i\in \N^+$. Without loss of generality, we assume that $R^i_{\lambda^*}$ is a local minimizer of $U^*_{\lambda^*}$.
Thus, from the equation in \eqref{KSwhole}, we infer that $(U^*_{\lambda^*})''(R^i_{\lambda^*}) = M > 0$. By continuity, there is $\varepsilon_0 > 0$ such that 
\begin{equation}\label{clbu}
(U^*_{\lambda^*})''(r) \geq \frac{M}{2}  \qquad r \in I_0 := [R^i_{\lambda^*} - \varepsilon_0, R^i_{\lambda^*} + \varepsilon_0] \,.
\end{equation} 
Now fix $\lambda>0$ sufficiently close to $\lambda^*$. We are going to estimate $R^i_\lambda$. Without loss of generality assume
$(U^*_{\lambda})'(R^i_{\lambda^*}) \leq 0$, the other case is analogous.  For any 
$r \in I_0^+ := [R^i_{\lambda^*}, R^i_{\lambda^*} + \varepsilon_0]$,  the mean value theorem implies
\begin{equation}\label{pvol}
(U^*_{\lambda})'(r) - (U^*_{\lambda})'(R^i_{\lambda^*}) = (U^*_{\lambda})''(\zeta) (r - R^i_{\lambda^*}) = 
  (U^*_{\lambda^*})''(\zeta) (r - R^i_{\lambda^*})  + (U^*_{\lambda} - U^*_{\lambda^*})''(\zeta) (r - R^i_{\lambda^*}) \,.
\end{equation}
for some $\zeta \in I_0$ depending on $\lambda$.
Setting $r = R^i_{\lambda^*} + \varepsilon_0$ in \eqref{pvol}, one has by Lemma \ref{lrgl} combined with \eqref{clbu} and the fact that
$(U^*_{\lambda^*})'(R^i_{\lambda^*}) = 0$ that 
\begin{equation}
(U^*_{\lambda})'(R^i_{\lambda^*} + \varepsilon_0) \geq \frac{M}{2}\varepsilon_0 - C_N |\lambda - \lambda^*| \,.
\end{equation}
Consequently, for $\lambda$ sufficiently close to $\lambda^*$, we deduce that $(U^*_{\lambda})'(R^i_{\lambda^*}) \leq 0 \leq (U^*_{\lambda})'(R^i_{\lambda^*} + \varepsilon_0)$ and
by the intermediate value theorem, there is $j$ such that $R^j_{\lambda} \in [R^i_{\lambda^*}, R^i_{\lambda^*} + \varepsilon_0]$. 
Since critical points of $U_{\lambda^*}^*$ are non-degenerate, it is standard to see that $j = i$.
Therefore, by setting $r = R_\lambda^i$ in \eqref{pvol}, one obtains
\begin{equation}
0 \leq \frac{M}{2} (R_\lambda^i - R^i_{\lambda^*}) \leq (U^*_{\lambda^*})''(\zeta) (R_\lambda^i - R^i_{\lambda^*}) 
\leq C_{\lambda^*} |\lambda - \lambda^*|,
\end{equation}
where we used the fact that $(U^*_{\lambda^*})'(R^i_{\lambda^*}) = (U^*_{\lambda})'(R^i_{\lambda}) = 0$ in the last inequality. 
This concludes the proof. 
\end{proof}

We are now in position to prove Theorem \ref{generic}.

\begin{proof}[Proof of Theorem \ref{generic}]
Fix $i\in \N^+$. Since by Lemma \ref{lmlip} the function $\lambda  \mapsto F_i(\lambda) := R^i_\lambda$ is Lipschitz, 
 Rademacher's theorem implies that $F_i$ is differentiable 
for any $\lambda\in(0, \infty) \setminus S_i$, where $S_i$ has Lebesgue measure zero. 
Denote by $E_i =\{\lambda \in (0, \infty) \setminus S_i : F_i'(\lambda) = 0 \}$. 
Then, by Sard's theorem for Lipschitz functions (see \cite{varberg}), one has that $F_i(E)$ is a set of measure zero. 
Moreover, since $F_i$ is 
Lipschitz with locally uniformly bounded Lipschitz constant, one has that $F_i(S_i)$ has also measure zero. 

Overall, $S^*_i := F_i(S_i \cup E_i)$ is a set of zero measure, and therefore $S^* = \bigcup_i S^*_i$ has measure 
zero as well. Thus for any radius $R \in (0, \infty) \setminus S^*$, any $i$ and any 
$\lambda^*$ such that $R^i_{\lambda^*} = R$, the function
$\lambda \mapsto R^i_{\lambda}$ is differentiable at $\lambda^*$ with nonzero derivative. 
We claim that for any $\lambda$ sufficiently close to $\lambda^*$ one has $(U^*_{\lambda})'(R) \neq 0$. 
Indeed, otherwise there is a sequence $\lambda_n \to \lambda^*$ such that $(U^*_{\lambda_n})'(R) \neq 0$, or equivalently, 
$R^{i_n}_{\lambda_n} = R$ for any $n \geq 1$. Since critical points of $U^*_{\lambda^*}$ are non-degenerate, $i_n = i$ for 
any sufficiently large $n$. Then, by the definition of the derivative $F_i'(\lambda^*)=\partial_\lambda R^i_{\lambda^*} = 0$, a contradiction. 
\end{proof}

Next, we turn to the proof of Theorem \ref{generic2}. The main ingredient of the proof is the fact that for some compact interval $I \subset (0,\infty )$, the function $\lambda \rightarrow r^i_{\lambda ,\gamma}$ is bounded in $C^2 (I)$ by a constant that does not depend on $\gamma$. To show this, we first prove that $\partial_\lambda u(\cdot ,\gamma ,\lambda)$ and $\partial^2_\lambda u(\cdot ,\gamma ,\lambda)$ are uniformly bounded in $\gamma$.

\begin{lem}\label{lguc}
For any compact interval $I \subset [0, \infty)$,
the function $\lambda \mapsto u(\cdot, \gamma, \lambda)$ is a locally $C^2$ map from $(0, 1)$ to $C^2(I)$, 
where $u(\cdot ,\gamma ,\lambda)$ is the solution to \eqref{KSeqintfin}. Moreover, 
the derivatives (in $\lambda$, up to second order) are bounded uniformly in $\gamma$.
\end{lem}

\begin{proof}
Due to the smooth dependence on data, the function $\lambda \mapsto u(\cdot, \gamma, \lambda)$ is smooth, so the main challenge
is to prove the uniform boundedness of derivatives. 

Fix $\gamma > 0$, $\lambda_1, \lambda_2 > 0$ and denote by $u_i$, $i \in \{1, 2\}$ the solution of \eqref{KSeqintfin} with 
$\lambda = \lambda_i$ and $\gamma = \gamma$. Also, fix a compact interval $I \subset [0, \infty)$. As in Section \ref{secconv},
denote by 
$\hat{u}_i (\rho) = u_i(r) - \gamma$ with $\rho = e^{\frac{\gamma}{2}}r$ and note that $\hat{u}_i$ solves \eqref{defuhat} with
$\lambda = \lambda_i$. It is easy to see that $-\gamma \leq \hat{u}_i \leq 0$ and by the continuous dependence on the initial data, 
one has $\hat{u}_2 \to \hat{u}_1$ in $C^2[0, A)$, for any fixed $A > 0$.  

If we set $w(\cdot ,\lambda_1 ,\lambda_2)=w(\cdot)  = \hat{u}_1 (\cdot ) - \hat{u}_2 (\cdot )$, then $w$ satisfies
\begin{equation}
-w'' - \frac{N-1}{r} w' + e^{-\gamma} w = \lambda_1 e^{\hat{u}_1}(1 - e^{-w}) + (\lambda_1 - \lambda_2)e^{\hat{u}_2}, \qquad 
w(0) = w'(0) = 0 \,.
\end{equation}
Also, by choosing $\lambda_2$ sufficiently close to $\lambda_1$, we can assume that $|w| \leq 1$ on $[0, A)$. Since $\hat{u}_i \leq 0$, then
$e^{\hat{u}_i} \leq 1$. Hence, we have
\begin{equation}
\lambda_1 e^{\hat{u}_1}(1 - e^{-w}) =\lambda_1 e^{\hat{u}_1} \left( \frac{1 - e^{-w} - w}{w} + 1 \right)w = m(r) w 
\end{equation}
and for a universal constant $C$ one has 
$|m(r)| \leq C \lambda_1 $ since $|w| \leq 1$ and $\hat{u}_1 \leq 0$. If we denote 
$\bar{w}(\cdot ,\lambda_1 , \lambda_2 ) = \frac{1}{\lambda_1 - \lambda_2} w(\cdot ,\lambda_1 , \lambda_2 )$, then $\bar{w}$ satisfies
\begin{equation}\label{elbi}
-\bar{w}'' - \frac{N-1}{r} \bar{w}' + e^{-\gamma} \bar{w} =  m(r)\bar{w} + e^{\hat{u}_2}, \qquad 
\bar{w}(0) = \bar{w}'(0) = 0 \,.
\end{equation} 
Since all coefficients are bounded, we obtain that $\|\bar{w}\|_{C^2[0, A)} \leq C_{\lambda_1}$, where $C_{\lambda_1}$ is bounded in 
$\lambda_1$ on bounded intervals. Equivalently, 
\begin{equation}\label{liwu}
\|\hat{u}_1 - \hat{u}_2\|_{C^2[0, A)} = \|w\|_{C^2[0, A)} \leq C_{\lambda_1} |\lambda_1 - \lambda_2|\,.
\end{equation}
By the smooth dependence on data (the solutions are regular), 
the function $\lambda \mapsto \hat{u}(\cdot, \lambda_1)$ is differentiable as a map from real 
numbers to $C^2_{\textrm{loc}}[0, A)$ and by \eqref{liwu} its derivatives are bounded, uniformly in $\lambda$. 
Next, we prove that the derivatives are in fact Lipschitz continuous in $\lambda$. For that purpose, we explicitly
indicate the dependence of $w$ and $\bar{w}$ on $\lambda$. 

Fix $\lambda_1 > 0$. By Arzel\` a-Ascoli's theorem, we have 
\begin{equation}
\partial_\lambda w(\cdot, \lambda_1) = \lim_{\lambda \to \lambda_1} \bar{w}(\cdot, \lambda_1, \lambda) =: Z(\cdot, \lambda_1) \,,
\end{equation}
where the convergence is in $C^2[0, A)$. Fix $\lambda_1, \lambda_2 \in (0, \infty)$ and denote $Z_i := Z(\cdot, \lambda_i)$, 
$i \in \{1, 2\}$. By passing to the limit in \eqref{elbi}, we obtain that $Z_i$ solves 
\begin{equation}
-Z_i'' - \frac{N-1}{r} Z_i' + e^{-\gamma} Z_i =  m_i(r)Z_i + e^{\hat{u}_{i}}, \qquad 
Z_i(0) = Z_i'(0) = 0 \,,
\end{equation}
where we used that $e^{\hat{u}_2} \to e^{\hat{u}_1}$ as $\lambda_2 \to \lambda_1$ and, since $w \to 0$, 
\begin{equation}
m_i(r) = \lim_{\lambda \to \lambda_i} \lambda_i e^{\hat{u}_i} \frac{1 - e^{-w} - w}{w} + \lambda_i e^{\hat{u}_i} = 
 \lambda_i e^{\hat{u}_i} \,.
\end{equation}
Analogously as above, by defining $\bar{Z} = (Z_1 - Z_2)/(\lambda_1 - \lambda_2)$, we obtain that $\bar{Z}$ solves
\begin{equation}\label{bzef}
-\bar{Z}'' - \frac{N-1}{r} \bar{Z}' + e^{-\gamma} \bar{Z} =  m_1(r)\bar{Z} + \frac{m_1(r) - m_2(r)}{\lambda_1 - \lambda_2} Z_2+ 
\frac{e^{\hat{u}_{1}} - e^{\hat{u}_{2}}}{\lambda_1 - \lambda_2}, \qquad 
\bar{Z}(0) = \bar{Z}'(0) = 0 \,.
\end{equation}
Furthermore, since $\hat{u}_i \leq 0$ by the mean value theorem, we have 
$|e^{\hat{u}_{1}} - e^{\hat{u}_{2}}| \leq |\hat{u}_{1} - \hat{u}_{2}|$ and since $\bar{w}$ is bounded, we infer that
the inhomogeneous term in \eqref{bzef} is bounded. Consequently, we obtain that $\bar{Z}$ is bounded in $C^2[0, A)$, and therefore 
\begin{equation}
\|\partial_\lambda w(\cdot, \lambda_1) - \partial_\lambda w(\cdot, \lambda_2)\|_{C^2[0, A)} = 
\|Z_1 - Z_2\|_{C^2[0, A)} \leq C |\lambda_1 - \lambda_2|, 
\end{equation}
where the constant $C$ is locally uniform in $\lambda_1$ and $\lambda_2$. Returning to the original variables, we obtain that
\begin{equation}
\|\partial_\lambda u(\cdot, \lambda_1, \gamma) - \partial_\lambda u(\cdot, \lambda_2, \gamma)\|_{C^2[0, Ae^{\gamma/2})} \leq 
\|\partial_\lambda \hat{u}(\cdot, \lambda_1, \gamma) - \partial_\lambda \hat{u}(\cdot, \lambda_2, \gamma)\|_{C^2[0, A)}
\leq C|\lambda_1 - \lambda_2| \,,
\end{equation}
where $C$ is independent of $\gamma > 0$ and locally uniformly bounded in $\lambda_1, \lambda_2 > 0$ as desired. 
\end{proof}

\begin{lem}\label{rglb}
For any sufficiently large $\gamma > 0$, the function $\lambda \mapsto r^{i}_{\lambda,\gamma}$
belongs to $C^2_{\textrm{loc}}(0, \infty)$, where $ r^{i}_{\lambda,\gamma}$ is defined as in Theorem \ref{KSoscbranch}. Furthermore, for any compact interval $I \subset (0, \infty)$, we have
\begin{equation}
\| \lambda \mapsto r^{i}_{\lambda,\gamma}\|_{C^2(I)} \leq C_{I} \,,
\end{equation} 
where $C_I$ is in particular independent of $\gamma$. 
\end{lem}

\begin{proof}
First observe that Theorem \ref{KSsingintro} and standard elliptic theory implies  
$u(\cdot, \gamma) \to U^*$ in $C^2_{\textrm{loc}}$. Therefore by Lemma \ref{liouville} there is $\delta_1 > 0$ such that, 
for sufficiently large
$\gamma$, one has $u( r^{i}_{\lambda,\gamma}) \in (\underline{u}_\lambda + \delta_1, \bar{u}_{\lambda} - \delta_1)$, 
$u'( r^{i}_{\lambda,\gamma}) = 0$  
and
$|u''( r^{i}_{\lambda,\gamma})| \geq c_{\lambda} > 0$, 
where in particular $c_{\lambda}$ does not depend on $\gamma$. Then, since the arguments in the proof of Lemma \ref{lmlip} are local around the critical points, by using 
Lemma \ref{lguc}, we similarly obtain that 
$\lambda \mapsto  r^{i}_{\lambda,\gamma}$ is Lipschitz with Lipschitz constant independent of $\gamma$ and locally 
bounded in $\lambda$.  

In fact, due to the smooth dependence of solutions on data, the function $\lambda \mapsto  r^{i}_{\lambda,\gamma} $ is $C^1$ with
derivative bounded independently of large $\gamma$. 
Let us prove that $\partial_\lambda  r^{i}_{\lambda,\gamma}$ has 
also locally bounded derivative. Define $v (r) = u( r^{i}_{\lambda,\gamma} r , \gamma)$ and note that $v$ satisfies
\begin{equation}\label{svea}
- \frac{1}{( r^{i}_{\lambda,\gamma})^2} v'' - \frac{N - 1}{ r^{i}_{\lambda,\gamma} r} v' + v = \lambda e^v, \qquad
v(0) = \gamma, \quad v'(0) = v'(1) = 0 \,.
\end{equation}
If we denote $\tilde{w} = \partial_\lambda v$ and differentiate \eqref{svea} with respect to $\lambda$, we obtain
\begin{equation}
\frac{\partial_\lambda  r^{i}_{\lambda,\gamma}}{ r^{i}_{\lambda,\gamma} }(v - \lambda e^v)
=
\partial_\lambda  r^{i}_{\lambda,\gamma}
\left( \frac{1}{( r^{i}_{\lambda,\gamma})^3} v'' + \frac{N - 1}{( r^{i}_{\lambda,\gamma})^2 r} v'
 \right) = \frac{1}{( r^{i}_{\lambda,\gamma})^2} \tilde{w}'' + \frac{N - 1}{ r^{i}_{\lambda,\gamma} r} \tilde{w}' - \tilde{w}
 + \lambda e^v \tilde{w} + e^v \,.
\end{equation}
If we return to the original variables, we have
\begin{equation}
\frac{\partial_\lambda r^{i}_{\lambda,\gamma}}{ r^{i}_{\lambda,\gamma} }(u - \lambda e^u) = 
 (\partial_\lambda u)'' + \frac{N - 1}{r} (\partial_\lambda u)' + (\partial_\lambda u)(\lambda e^u - 1) + e^u \,.
\end{equation}
Substituting $r = r^{i}_{\lambda,\gamma}$ in the previous line, one obtains
\begin{equation}\label{hfi}
\partial_\lambda  r^{i}_{\lambda,\gamma} =  r^{i}_{\lambda,\gamma}
\left.\frac{(\partial_\lambda u)'' + \frac{N - 1}{r} (\partial_\lambda u)' + (\partial_\lambda u)(\lambda e^u - 1) + e^u}
{u - \lambda e^u}\right|_{r = r^{i}_{\lambda,\gamma}}
\end{equation}
As proved above, $\lambda \mapsto  r^{i}_{\lambda,\gamma}$ is Lipschitz with Lipschitz constant uniform in $\gamma$. 
Also, by Lemma \ref{lguc} the function $\lambda \mapsto u(\cdot, \lambda)$ is $C^2((0, \infty), C^2_{\textrm{loc}}(0, \infty))$ 
with second derivative 
bounded uniformly in $\gamma$. Finally, for $\gamma$ sufficiently large (cf. choice of $\delta_1$ above), one has 
that the denominator on the right hand side of \eqref{hfi} is bounded away from zero independently of $\gamma$. 
Overall, the right had side of \eqref{hfi} is Lipschitz with constant bounded independently of $\gamma$, 
and the lemma follows. 
\end{proof}

We are now in position to prove Theorem \ref{generic2}.

\begin{proof}[Proof of Theorem \ref{generic2}]
Looking for a contradiction, assume that there exist  sequences $\gamma_n \to \infty$ and $\lambda_n, \lambda'_n \to \lambda^i$
such that $ r^{i}_{\lambda_n,\gamma_n} =  r^{i}_{\lambda_n^\prime,\gamma_n} = R$. Then, by the mean value theorem, 
there exists $\lambda^*_n$ between $\lambda_n$ and $\lambda_n'$ such that 
$\partial_\lambda  r^{i}_{\lambda_n^\ast,\gamma_n} = 0$. Since $\lambda_n^* \to \lambda^i$ and 
$\lambda \mapsto \partial_\lambda  r^{i}_{\lambda,\gamma_n} $ has bounded (in $\gamma$ and locally in $\lambda$) 
second derivative, one has that 
$\partial_\lambda   r^{i}_{\lambda^i,\gamma_n}  \to 0$ as $n \to \infty$. 

Furthermore, by Lemma \ref{rglb}, the sequence of functions $( \lambda \mapsto  r^{i}_{\lambda_n,\gamma_n})_n$
is uniformly bounded in $C^{1, 1}_{\textrm{loc}}$, and therefore by Arzel\` a-Ascoli's theorem, it converges in 
$C^{1}_{\textrm{loc}}$. In addition $( \lambda \mapsto  r^{i}_{\lambda_n,\gamma_n})_n$ converges pointwise to 
$R^i_{\lambda^i}$, and therefore it converges to $R^i_{\lambda^i}$ in $C^{1}_{\textrm{loc}}$. 

Combining the previous observations, we obtain that $\partial_\lambda R^i_{\lambda^i} = 0$, a contradiction
to the definition of the set $S^*$. 
\end{proof}

\bibliographystyle{plain}

\end{document}